\documentclass{siamltex}
\usepackage{amsmath,amsfonts,amssymb}
\usepackage{latexsym}
\usepackage{graphicx,overpic}
\usepackage{hyperref}
\usepackage{caption, subcaption}
\usepackage{pgfplots}
\usepgfplotslibrary{groupplots}

\pgfplotsset{compat=newest}
\usepackage{color}
\usepackage{booktabs}
\usepackage{tikz}
\usepackage{todonotes}
\usepackage[shortlabels]{enumitem}
\usepackage{placeins}

\newcommand{\R}{\mathbb R}

\newcommand{\bE}{\mathbf E}
\newcommand{\bF}{\mathbf F}

\newcommand{\bH}{\mathbf H}

\newcommand{\bV}{\mathbf V}

\newcommand{\bgp}{{\hat \bz}_\perp}
\newcommand{\bn}{\mathbf n}
\newcommand{\be}{\mathbf e}

\newcommand{\bp}{\mathbf p}
\newcommand{\bq}{\mathbf q}
\newcommand{\br}{\mathbf r}

\newcommand{\bu}{\mathbf u}
\newcommand{\bv}{\mathbf v}
\newcommand{\bw}{\mathbf w}
\newcommand{\bx}{\mathbf x}

\newcommand{\bz}{\mathbf z}

\newcommand{\K}{\mathcal K}

\newcommand{\T}{\mathcal T}

\newcommand{\cK}{\mathcal K}
\newcommand{\cO}{\mathcal O}
\newcommand{\cT}{\mathcal T}

\newcommand{\Div}{\mathop{\rm div}}

\newcommand{\hatbz}{{\hat \bz}}

\DeclareGraphicsExtensions{.pdf,.eps,.ps,.eps.gz,.ps.gz,.eps.Y}

\newcommand{\bxi}{\mbox{\boldmath$\xi$\unboldmath}}
\newcommand{\Kzh}{\cK_{\hat \bz,h}}
\newcommand{\Kz}{\cK_{\hat \bz}}

\newtheorem{assumption}{Assumption}[section]

\newtheorem{remark}{Remark}[section] 
\begin{document}
\title{Numerical analysis of a constrained strain energy minimization problem}
\author{Tilman Aleman\thanks{Institut f\"ur Geometrie und Praktische  Mathematik, RWTH-Aachen
University, D-52056 Aachen, Germany (aleman@igpm.rwth-aachen.de)} \and Arnold Reusken\thanks{Institut f\"ur Geometrie und Praktische  Mathematik, RWTH-Aachen
University, D-52056 Aachen, Germany (reusken@igpm.rwth-aachen.de)}
}
\maketitle
%\tableofcontents
%
\begin{abstract}
We consider a setting in which an evolving surface is implicitly characterized as the zero level of  a level set function. Such an implicit surface does not encode any information about the path of a single point on the evolving surface. In the literature different approaches for determining a velocity that induces corresponding paths of points on the surface have been proposed. One of these is based on minimization of the strain energy functional. This then leads to a constrained minimization problem, which has a corresponding equivalent formulation as a saddle point problem. The main topic of this paper is a detailed analysis of this saddle point problem and of a finite element discretization of this problem. We derive well-posedness results for the continuous and discrete problems and optimal error estimates for a finite element discretization that uses standard $H^1$-conforming finite element spaces. 
\end{abstract}

\section{Introduction}
The problem that we consider in this paper is inspired by an application from  computer graphics.
In many computer graphics pipelines implicit representations of geometry are used. Typically implicit shapes are represented as level sets of a time dependent level set function. In such a setting the temporal coherence is rather weak, in the sense that such implicit surfaces do not encode any information about the path of a single point on the evolving surface. Therefore, the problem of determining a ``reasonable'' velocity field corresponding to an implicity evolving surface has been addressed in the computer graphics community \cite{stam2011velocity,Ben-Chen_2010_CGOM,Tao2016}. 

%Typically, surfaces are represented either through explicit approximations such as surface meshes or via implicit methods such as level sets. When considering moving surfaces, applying a motion to a surface mesh is straightforward, as one is able to just apply the velocity field to the vertices of the mesh. Explicit approaches usually have the disadvantage of potential remeshing and reinitialization requirements, especially when topological changes need to be considered. On the other hand, topological changes of surfaces are easily handled by level set methods but when a moving surface is only given by a (time-dependent) family of level set functions, only global information of the surface transport is available. This means that local information is not encoded in the level set function. So, more involved approaches are needed when local surface properties are central to the application.

One natural approach, based on computing a ``near-isometric'' velocity field (also called approximate Killing vector field) is introduced in  \cite{Tao2016}. The techique is also used in \cite{Slavcheva_2018_diss, atzmon2021augmenting}.  Similar near-isometric velocity fields are used in a physics informed  neural network method for diffeomorphic shape or image registration \cite{dummer2024} and  in an optimization technique for mesh parametrization \cite{claici2017}. We outline the main idea of this approach, cf. Section~\ref{sectLS} for more details. Assume that the  surface one is interested in is implicitly represented as the zero level of a \emph{known} level set function $\phi(\bx,t)$, $\bx \in \Omega \subset \R^d$, $d=2,3$, $t \in [0,t_e]$. Here $\Omega$ is a fixed  domain that contains an evolving zero level of $\phi(\cdot,t)$ for $t \in [0,t_e]$. 
%We assume $\nabla \phi \neq 0$ on $\Omega \times [0,t_e]$. 
The level sets of $\phi$ with value $c$ are denoted by $\Gamma_c(t)$ and the zero level is denoted by $\Gamma(t)=\Gamma_0(t)$. 
The evolution of the level set $\Gamma_c(t)$ is completely determined by its velocity in \emph{normal} direction, denoted by $\bu_N=u_N \bn$, with $\bn$ the unit normal vector on $\Gamma_c(t)$. An arbitrary (smooth) velocity field $\bv=\bv(\cdot,t)$ on $\Omega$ can be uniquely decomposed based on the normal and tangential components on each of the level  sets $\Gamma_c(t)$. This decomposition is denoted by $\bv=\bv_N +\bv_T$. Now let $\bv$ be such that its normal component coincides with $\bu_N$, i.e., $\bv=\bu_N +\bv_T$. Such a velocity field is consistent with $\phi$ in the sense that it induces trajectories of points that stay on $\Gamma(t)$. More precisely, for $\bxi_0 \in \Gamma(0)$ let $\bxi(t)$  be such that $\bxi(0)= \bxi_0$, $\frac{d}{dt}\bxi(t)=\bv(\bxi(t),t)$, $t \in [0,t_e]$, then $\bxi(t) \in \Gamma(t)$ holds. Of course the trajectories $\bxi(t)$ depend on the choice of the tangential component $\bv_T$. In \cite{Tao2016} it is proposed to determine this tangential component based on minimization of a strain energy functional as follows:
\begin{equation} \label{eqprob1}
  \bv_T = {\rm argmin}_{\bw_T} \int_\Omega \bE(\bu_N+ \bw_T): \bE(\bu_N+ \bw_T)\, d\bx,\quad \bE(\bv):=\nabla \bv +\nabla \bv^T, 
\end{equation}
where the vector fields $\bw_T=\bw_T(\cdot,t)$ must be tangential to the level sets $\Gamma_c(t)$. Since a vector field $\bv$ that satisfies $\nabla \bv +\nabla \bv^T=0$ (on $\Omega$) is called a Killing vector field, a velocity field $\bv= \bu_N+\bv_T$, with $\bv_T$ a solution of \eqref{eqprob1}, is called an approximate Killing vector field, cf. \cite{Tao2016}. 
Hence, for each fixed $t \in [0,t_e]$ we obtain a constrained minimization problem. It is natural to introduce a Lagrange multiplier and reformulate  the problem \eqref{eqprob1} as an unconstrained  saddle point problem. The main topic of this paper is a detailed analysis of this saddle point problem and of a finite element discretization of this problem. The saddle point reformulation is also used in \cite{Tao2016}, but  as far as we know there is no literature in which well-posedness of this problem or error analysis of a discretization method is studied.

We present a well-posed saddle point formulation of the problem \eqref{eqprob1} in natural Sobolev spaces. We introduce a finite element discretization of the saddle point problem using an $H^1$-conforming pair of finite element spaces with polynomials of degree $k$ both for the (tangential) velocity and the Lagrange multiplier. We prove discrete LBB stability of this pair and derive optimal discretization error bounds. We also include results of numerical experiments that confirm the error analysis and we present results that show how such a velocity field is used to determine the trajectories of points on a evolving $\Gamma(t)$. A technical issue that we have to deal with is that of a possible nonuniqueness of a solution of \eqref{eqprob1}, which is related to rigid motions.

The remainder of the paper  is organized as follows. In Section~\ref{sectLS} we give a more detailed derivation of the constrained minimization problem \eqref{eqprob1}. This problem is formulated in a slightly more general form  in Section~\ref{Secproblem}. In that section we also study the kernel space that has to be factored out to obtain a uniquely solvable problem. In Section~\ref{sect3} we introduce a saddle point formulation and prove well-posedness of this formulation. A finite element discretization of this saddle point problem is explained in Section~\ref{secdiscr} and an error analysis of this method is presented in Section~\ref{secerroranalysis}. Finally, in Section~\ref{Experiments} we present results of some numerical experiments. 

\section{Motivation: An inverse level set problem} \label{sectLS}
The problem that we study is motivated by an application of level set techniques in computer graphics \cite{Tao2016}. We consider a situation in which a smooth  evolving  hypersurface  $\Gamma(t)$ (curve in $\R^2$, surface in $\R^3$) is implicitly represented by a level set function $\phi(\bx,t)$, i.e., $\phi(\bx,t)=0$ iff $\bx \in \Gamma(t)$.  As is common in level set techniques, we use a fixed  computational open domain $\Omega \in \R^d$, $d=2,3$, time interval $[0,t_e]$, and consider the evolution  of $\Gamma(t) \cap \Omega$ for $t \in [0,t_e]$ based on the transport of a level set function $\phi: \, \Omega_T:=\overline{\Omega} \times [0,t_e] \to \R$. 
We assume
\begin{equation} \label{assphi} \phi \in C^1(\Omega_T), \quad \min_{(\bx, t) \in \Omega_T}\|\nabla \phi(\bx,t)\| \geq c_{\min} >0.
 \end{equation}
 Here $\|\cdot\|$ denotes the Euclidean norm in $\R^d$. 
For a fixed $t\in [0,t_e]$, define $\phi_{\min}= \min_{\bx \in \overline{\Omega}}\phi(\bx,t)$, $\phi_{\max}= \max_{\bx \in  \overline{\Omega}}\phi(\bx,t)$ and the level sets $\Gamma_c(t):=\{\, \bx \in \Omega~|~ \phi(\bx,t)=c\,\}$, $\phi_{\min} \leq c \leq \phi_{\max}$. To avoid some technical details, we assume ${\rm meas}_{d-1}\Gamma_c(t) > 0$. Note that $\Omega= \cup_{\phi_{\min} \leq c \leq \phi_{\max}} \Gamma_c(t)$ holds. From the implicit function theorem it follows that the  level sets $\Gamma_c(t)$ have $C^1$ smoothness. For the unit normal on $\Gamma_c(t)$ (pointing in the direction of increasing $\phi$) we have 
\[
  \bn(\bx,t)= \frac{\nabla \phi(\bx,t)}{\|\nabla \phi(\bx,t)\|}, \quad \bx \in \Gamma_c(t).
\]
Note that this \emph{normal vector field} $\bn$ is defined on $\Omega_T$. The transport of the level set function can be described using the level set equation
\begin{equation} \label{LSeq}
 \frac{\partial \phi}{\partial t} + \bu \cdot \nabla \phi =0 \quad \text{on}~\Omega_T,
\end{equation}
where $\bu$ has to satisfy 
\begin{equation} \label{condu} \bu\cdot \bn = - \frac{\partial \phi }{\partial t}\frac{1}{ \|\nabla \phi\|} \quad \text{on}~\Omega_T.
\end{equation}
Trajectories  induced by such a velocity field $\bu$ have the property that transport conserves level set values.  More precisely, 
for arbitrary $\underline{t} \in [0,t_e)$ and $\bxi_0 \in \Gamma_c(\underline{t})$, consider the trajectory given by $\bxi(\underline{t})= \bxi_0$, $\frac{d}{dt}\bxi(t)= \bu(\bxi(t),t)$, $t \in  [\underline{t}, t_e)$. Then it holds that $\bxi(t) \in \Gamma_c(t)$ for $t \in [\underline{t},\underline{t}+ \delta]$ for sufficiently small $ \delta >0$. This in particular holds for points $\bxi_0$ on the zero level $\Gamma_0(t)$. Clearly, the trajectory depends on the specific choice of $\bu$, for which there is freedom, since the condition \eqref{condu} is satisfied for any $\bu$ such that
\begin{equation} \label{struct} 
\bu = - \frac{\partial \phi }{\partial t} \frac{1}{ \|\nabla \phi\|}\bn +\bv, \quad   \quad \text{with}~\bv ~ \text{such that}~\bv\cdot \bn=0 
\end{equation}
holds. This can be considered as a sort of inverse level set problem, where for a given  level set function $\phi$ on $\Omega_T$ we want to determine a ``suitable'' velocity $\bu$ that satisfies \eqref{struct}. This velocity field then induces corresponding trajectories of points. 

\begin{remark} \label{ex1} \rm Consider a simple example for $d=2$:
\begin{equation} \label{exsimle}
  \phi(\bx,t)= \phi(x,y,t)= (x-t)^2 + \tfrac12 y^2 -1, \quad t\in [0,\tfrac14], ~\Omega =(-2,2)^2 \setminus [-\tfrac12, \tfrac12]^2.
\end{equation}
For the minimum and maximum values of $\phi(\cdot,t)$ we have $\phi_{\min}=(\tfrac12-t)^2-\tfrac78$ and $\phi_{\max}=(2+t)^2+1$, respectively. 
The level set $\Gamma_c(t)$ is an ellipse with centre $(t,0)$ and axes $\sqrt{1+c}$, $\sqrt{2(1+c)}$, intersected with $\Omega$. Note that for $(\bx,t) \in \Omega_T$ we have $\|\nabla \phi(\bx,t)\| =  \sqrt{4(x-t)^2 + y^2} \geq 2 |x-t| \geq \tfrac12$ and  the normal vector field is given by
\[
  \bn(\bx,t)= \frac{1}{\sqrt{4(x-t)^2+ y^2}} \begin{pmatrix} 2(x-t) \\ y \end{pmatrix}.
\]
This level set function describes the transport of an ellipse with velocity field $\bu^\ast=( 1~0)^T$, i.e., a rigid motion. This level set function solves \eqref{LSeq} with $\bu =\bu^\ast$. It is, however, not obvious how this particular choice of $\bu^\ast$ can be determined if only the level set function information is available.  
\end{remark}

In the next section we describe, in a slightly more general setting,  a natural criterion that uniquely determines the remaining tangential velocity part $\bv$ in \eqref{struct} . This approach is introduced in \cite{Tao2016}.\\
%({\bf AR}: for each fixed $t$ the problem \eqref{struct} is considered; what about regularity of resulting $\bu(\cdot,t)$ as function of $t$? This is not analyzed.)
%
\section{A constrained strain energy minimization problem} \label{Secproblem}
In this section we formulate a class of  minimization problems. A minimization problem of this form is used to determine, for a fixed $t \in [0,t_e]$, the tangential velocity part $\bv(\cdot, t)$ in \eqref{struct}.
%Let $\Omega \subset \R^d$, $d=2,3$, be a connected domain with Lipschitz boundary. 
The usual Sobolev spaces of once weakly differentiable scalar or vector valued functions on $\Omega$ are denoted by $H^1(\Omega)$, $\bH^1(\Omega):=H^1(\Omega)^d$.  For $\bv \in \bH^1(\Omega)$ we define $\bE(\bv):= \nabla \bv + \nabla \bv^T$ and the (strain) energy functional
\begin{equation} \label{functional}
  a(\bv,\bv):=\int_\Omega \bE(\bv):\bE(\bv)\, d \bx.
\end{equation}
The kernel of $a(\cdot,\cdot)$ is given by $K={\rm Ker}(\bE)=\{\, \bv \in \bH^1(\Omega)~|~a(\bv,\bv)=0\,\}$ and consists of the rigid motions. Its dimension is 3 or 6 for $d=2$ and $d=3$, respectively.
\begin{remark} \label{remBasis}
 \rm An $L^2$-orthonormal basis of $K$ for $d=2$ is given by 
\begin{equation} \label{basis} \begin{split}
 \bv_1 & =|\Omega|^{-\frac12} (1,0)^T,\quad \bv_2=|\Omega|^{-\frac12} (0,1)^T, \quad \bv_3=d_3 \begin{pmatrix} d_1-y \\ d_2 +x \end{pmatrix}, \\
  d_1 & = |\Omega|^{-1}\int_{\Omega} y \, d\bx, \quad d_2=-|\Omega|^{-1}\int_{\Omega} x \, d\bx, \\ d_3&=\big( \int_\Omega (d_1-y)^2 +(d_2+x)^2 \, d\bx \big)^{-\frac12}.
\end{split} \end{equation}
%The $L^2$-orthogonal projection $P_K:\bL^2(\Omega) \to K$ can be represented as $P_K \bv= \sum_{i=1}^3 (\bv,\bv_i)_{L^2} \bv_i$.
\end{remark}
\ \\
Let $\hat \bz:\Omega \to \R^d$ be a \emph{given} continuous velocity field with $\|\hat \bz\|=1$ on $\Omega$. An arbitrary  continuous vector field $\bu:\Omega \to \R^d$ can be uniquely decomposed as $\bu = (\bu\cdot\hat \bz)\hat \bz + \bv$, with $\bv\cdot \hat \bz=0$, i.e., $\bv(\bx)\cdot\hat{\bz}(\bx)=0$ for all $\bx \in \Omega$. We consider a setting in which for a given $\hat \bz$ the component $\bu \cdot  \hat \bz$ of  $\bu$ in direction $\hat \bz$ is given and the only degree of freedom left is the component of $\bu$ orthogonal to $\hat \bz$. We will use a minimization problem based on the energy functional \eqref{functional} and therefore we require $(\bu \cdot \hat \bz)\hat \bz \in \bH^1(\Omega)$. Sufficient for this to hold is $\bu \cdot  \hat \bz \in
H^{1,\infty}(\Omega)$ and $\hat \bz \in \bH^{1,\infty}(\Omega)$. Motivated by this we introduce the following assumption
\begin{assumption} \label{ass1} \rm 
 We assume a given vector field $\bz:= z \hat \bz: \, \Omega \to \R^d$ with
 \begin{equation} \label{Assump}
  z \in H^{1,\infty}(\Omega), \quad \hat \bz \in \bH^{1,\infty}(\Omega) ~\text{and}~ \|\hat \bz\|=1 ~\text{a.e on}~\Omega.
\end{equation}
\end{assumption}

In the remainder pointwise conditions such as ``$\bv \cdot\hat \bz=0$'' are always meant in the usual a.e. sense. In a velocity field of the form $\bu=\bz+\bv$,  for determining the component $\bv$ we use the criterion that the corresponding $\bu$ should be ``as close as possible to a rigid motion''. More precisely, we consider minimization of $a(\cdot,\cdot)$ in the space of velocities $\bu=\bz+\bv$, with  $\bv\cdot\hat \bz= 0$ on $\Omega$. 
 Define the space
\[
 %V_\bz:=\{\, \bu \in \bH^1(\Omega)~|~\bu \cdot\bz=|\bz|^2\,\}, 
 V_{\hat \bz}^0:=\{\, \bv \in \bH^1(\Omega)~|~\bv \cdot\hat \bz=0\,\}.
\]
%Define $K:=K_a \cap V_{\hat \bz}^0$. In the generic case we have $K=\{0\}$, but in certain special cases $\dim (K) \geq 1$ holds. For example, if $\bz=\bc$ is a constant vector then vectors orthogonal to $\bc$ are elements of $K$.
Hence, we want to minimize $a(\bz+\bv,\bz+\bv)$ for $\bv \in V_{\hat \bz}^0$. Note, however, that $a(\bz+\bv,\bz+\bv)=a(\bz+\bv +\bw,\bz+\bv+\bw)$ for all $\bw \in K$. Clearly, for uniqueness of a minimizer we have to factor out the kernel $K$. We introduce:
\begin{equation}\label{orthocond}  \begin{split}
   \cK_{\hat \bz} & :=K \cap V_{\hat \bz}^0= \{\, \bv \in K~|~\bv \cdot \hat \bz=0 \quad \text{a.e. on}~\Omega\,\},\\
   V_{\hat \bz}^0 \setminus \cK_{\hat \bz} & :=\{\, \bv \in V_{\hat \bz}^0~|~(\bv,\bw)_{1}=0 \quad \text{for all}~\bw \in \cK_{\hat \bz}\,\}, \end{split} 
\end{equation}
where $(\cdot,\cdot)_1$ is the usual scalar product on $\bH^1(\Omega)$, i.e., $\|\bv\|_1^2= \sum_{i=1}^d \big(\|v_i\|_{L^2(\Omega)}^2 +\|\nabla v_i\|_{L^2(\Omega)}^2 \big)$.  
This then leads to the following \emph{constrained strain energy    minimization problem} that we study in this paper:
\begin{equation} \label{problem1}
  \min_{\bu \in \bz + V_{\hat \bz}^0 \setminus \cK_{\hat \bz}} a(\bu,\bu)= \min_{\bv \in V_{\hat \bz}^0 \setminus \cK_{\hat \bz}} a(\bz+\bv,\bz+\bv).
\end{equation}
Due to symmetry this  minimization problem is equivalent to $\min_{\bv \in  V_{\hat \bz}^0 \setminus \cK_{\hat \bz}} a(\bv,\bv) + 2 a(\bz,\bv)$, and since  $a(\bv,\bv) > 0$ for all $\bv  \in V_{\hat \bz}^0 \setminus \cK_{\hat \bz}$, $\bv \neq 0$, we obtain  that $\bv^\ast$ solves \eqref{problem1} iff $\bv^\ast \in  V_{\hat \bz}^0 \setminus \cK_{\hat \bz}$ satisfies
\begin{equation} \label{variational}
  a(\bv^\ast,\bv)= -a(\bz,\bv) \quad \text{for all}~\bv \in V_{\hat \bz}^0\setminus \cK_{\hat \bz}.
\end{equation}
Note that ``for all $\bv \in V_{\hat \bz}^0\setminus \cK_{\hat \bz}$'' can be replaced  by ``for all $\bv \in V_{\hat \bz}^0$''.
\begin{lemma} \label{lemma1}
 The minimization problem \eqref{problem1} has a unique solution.
\end{lemma}
\begin{proof} We use standard arguments, cf. \cite{GR}. The bilinear form $a(\cdot,\cdot)$ is continuous on  $\bH^1(\Omega)$.
Using the Korn inequality, injectivity of $\bv \to \bE(\bv)$ on $\bH^1(\Omega) \setminus K$ and the Peetre-Tartar lemma (cf. \cite{Ern04}) it follows that $a(\cdot,\cdot)$ is elliptic on $\bH^1(\Omega)\setminus K$. 
The continuity and ellipticity properties also hold on the closed subspace $V_{\hat \bz}^0 \setminus \cK_{\hat \bz}$. 
\end{proof}

Note that the inverse level set problem explained in Section~\ref{sectLS} fits in the setting introduced above. We then have, for a fixed $t$,   the given velocity field $\bz= - \frac{\partial \phi }{\partial t} \frac{1}{ \|\nabla \phi\|}\bn= z \hat \bz$, with $\hat \bz =\bn$,  and  as criterion for chosing the vector field $\bv(\cdot,t)$ in \eqref{struct} we use the minimization problem~\eqref{problem1}.

\subsection{The kernel space} \label{seckernel}
We study the kernel space $\cK_{\hat \bz}$, cf \eqref{orthocond}.
We will see that ``in the generic case'', as explained below,  this kernel is trivial, i.e, contains only the zero function. We first consider the case of the inverse level set problem, where $\hat \bz= \bn$ holds, and then the general case.
\subsubsection{The inverse level set problem}
Consider the problem described in Section~\ref{sectLS}. There we have, for a fixed $t$,  $\hat \bz = \bn$ and at $\bx \in \Omega$ the vector $\bn(\bx)$ (we deleted the argument $t$)  is the unit normal on the level set $\Gamma_c$ that contains $\bx$. 
We have 
$\bv \in  \cK_{\hat \bz}$ iff $\bv \in K$ and $\bv(\bx) \cdot \bn(\bx) =0$  for all $\bx \in \Gamma_c$ and all $c \in [\phi_{\min},\phi_{\max}]$. This means that $\bv$ is  in the tangent space of  the level set $\Gamma_c$ for all $c \in [\phi_{\min},\phi_{\max}]$. 
%Take $c=0$, i.e. $\Gamma_0=\Gamma(t)\cap \Omega$ for some $t \in [0,t_e]$ and $\Gamma_0$ is a smooth  curve ($d=2$) or  surface ($d=3$). F
Furthermore, $\bv \in K$ iff $\nabla \bv +\nabla \bv^T =0$, which implies  $\nabla_{\Gamma_c} \bv + \nabla_{\Gamma_c} \bv ^T=0$, where $\nabla_{\Gamma_c} $ is the covariant derivative (for the hypersurface $\Gamma_c$) of the tangential field $\bv$. It follows that $\bv$ is a Killing vector field for the hypersurface $\Gamma_c$, which means that the (local) flow induced by $\bv$ is an isometry on $\Gamma_c$. Nonzero flows that are isometries are only possible for very special level set shapes, namely 
%The space of \emph{surface} Killing vector fields, denoted by $\cK_\Gamma$,  is nontrivial only for very special hypersurfaces $\Gamma_0 \cap \Omega$. For $d=2$ we have $\dim (\cK_\Gamma) \leq 1$ and $\dim (\cK_\Gamma) =1 $  only if the  curve $\Gamma_0\cap \Omega$ 
if $\Gamma_c$ a straight line segment or circular (for $d=2$) or  $\Gamma_c$ is planar or rotationally symmetric around a line (for $d=3$). We conclude that if for a $c \in [\phi_{\min},\phi_{\max}]$ the level set  $\Gamma_c$ does not have one of these special shapes, the kernel space  $ \cK_{\hat \bz}$ is trivial. 

\subsubsection{General case}
We now consider general vector fields $\hat \bz$ for which Assumption~\ref{ass1} holds. We consider the cases $d=2$ and $d=3$ seperately. 
\begin{lemma}\label{lemkernel1}
 Consider $d=2$. Then $\dim (\cK_{\hat \bz}) \leq 1$ holds. Define, with $\bx=(x,y)$, $R_\bx:=\begin{pmatrix} 1 & 0 & -y \\ 0 & 1 & x \end{pmatrix}$, $\bx \in \Omega$.   In the (generic) case in which 
\begin{equation} \label{CCC8}
 {\rm span} \{ \, \cup_{\bx \in \Omega} R_\bx^T \hat \bz(\bx)\} = \R^3
\end{equation}
is satisfied, we have $\dim (\cK_{\hat \bz}) = 0$. Let $\bw \in K$, $\bw \neq 0$, be a rigid motion. If $\bw$ is a rotation we assume that the center of rotation does not lie in $\overline{\Omega}$. Define $\hat \bz:= \frac{\bw_\perp}{\|\bw_\perp\|} \in \bH^{1,\infty}(\Omega)$, with $\bw_\perp \cdot \bw=0$. Then $\cK_{\hat \bz}={\rm span}(\bw)$ holds. 
\end{lemma}
\begin{proof}
See appendix. We note that the condition that the center of the rotation does not lie in $\overline{\Omega}$ is used to guarantee smoothness of $\hat \bz$.
\end{proof}
\ \\[1ex]
\begin{lemma}\label{lemkernel2}
 Consider $d=3$. Then $\dim (\cK_{\hat \bz}) \leq 3$ holds. Define, with $\bx=(x,y,z)$
\begin{equation} \label{defR}
  R_\bx= \begin{pmatrix}
          1 & 0 & 0& -y & -z & 0 
          \\ 0 & 1 & 0 & x & 0 & -z \\
          0 & 0 & 1 & 0 & x & y
         \end{pmatrix}.
\end{equation}
In the (generic) case in which 
\begin{equation} \label{CCC9}
 {\rm span} \{ \, \cup_{\bx \in \Omega} R_\bx^T \hat \bz(\bx)\} = \R^6
\end{equation}
is satisfied, we have $\dim (\cK_{\hat \bz}) = 0$. Take a fixed  $\bp \in \R^3$ with $\|\bp\|=1$.  Let $K_\bp$ be the three dimensional space of all rigid motions in the plane orthogonal to $\bp$, denoted by $\bp_\perp$. Let $P_{\bp_\perp}:\R^3 \to \bp_\perp$ be the orthogonal projection on $\bp_\perp$. A rigid motion $\bw \in K_\bp$, defined on $\bp_\perp$,  is extended to $\R^3$ by constant extension in normal direction, i.e., $\bw^e(\bx)=\bw( P_{\bp_\perp}\bx)$ for all $\bx \in \R^3$.   Then $\dim (\cK_{\hat \bz}) =3$ holds for
\begin{equation} \label{resdd} 
   \hat \bz =\bp,\quad 
   \cK_{\hat \bz} = \left\{ \bw^e~|~ \bw \in K_\bp\,\right\}.
 \end{equation}
\end{lemma}
\begin{proof}
See appendix. 
\end{proof}

The results above  explain in which sense the   kernel $\cK_{\hat \bz}$ is trivial ``in the generic case''.  
\begin{remark} \label{remrigid}
 \rm 
 If the kernel $\cK_{\hat \bz}$ is trivial then indeed, as claimed in \cite{Tao2016},  the minimization problem recovers rigid motions. Let $\bu $ be a rigid motion in $\R^d$ and $\bz= (\bu \cdot \hat \bz) \hat \bz$ with a given $\hat \bz$ that  satisfies Assumption~\ref{ass1} and for which $\cK_{\hat \bz}=\{\mathbf{0}\}$ holds. Then $\bv= \bu - \bz \in V_{\hat \bz}^0$ solves the minimization problem \eqref{problem1} because $a(\bz+\bv,\bz+\bv)=a(\bu,\bu)=0$ holds, and thus we recover $\bu=\bz+\bv$.  If $\dim (\cK_{\hat \bz}) \geq 1$, then in certain special cases a rigid motion is not recovered. As an example, consider $d=2$, and denote the rigid motions in coordinate directions by $\mathbf{e}_1(\bx)=( 1~0)^T$, $\mathbf{e}_2(\bx)=(0~1)^T$ for all $\bx \in \R^2$. Take $\bu = c_1 \mathbf{e}_1+ c_2 \mathbf{e}_2$ with constants $c_1,c_2\neq 0$. Take $\hat \bz  := \mathbf{e}_1$ and thus $\cK_{\hat \bz}={\rm span} (\mathbf{e}_2)$. Consider the minimization problem \eqref{problem1} with  $\bz= c_1 \hat \bz$. The rigid motion $\bu$ is recovered only for $\bv=c_2 \mathbf{e}_2$, which, however, is not feasible since $\mathbf{e}_2 \notin V_{\hat \bz}^0 \setminus K= V_{\hat \bz}^0 \setminus \cK_{\hat \bz}$.   
\end{remark}

%\begin{lemma} \label{lemkernel}
%Take $\bw \in K$ and $\bz \in C^1(\Omega)$. The following holds:
%\begin{align}
%  \big( \exists~x \in \Omega:~ \bw(x)\cdot \nabla |\bz(x)|^2 \neq 0\big)  & \Rightarrow ~ \bw \notin %V_{\bz}^0 \label{implic1}\\
%  \bw \in V_{\bz}^0  & \Rightarrow ~  \big(\bw\cdot \nabla |\bz|^2 = 0 ~ \text{on}~\Omega\big)\label{implic2}
%\end{align}
%\end{lemma}
%\begin{proof}
%Since $\bw \in K$ we have that $\bv^T \nabla \bw \, \bv = 0$ for all $\bv \in  \R^d$. Note that $\bw\cdot \nabla |\bz|^2 =\bz^T \nabla \bz \, \bw$ and thus we get
%\[
%  \bw\cdot \nabla |\bz|^2= \bz^T\big( \nabla \bz \,\bw + \nabla \bw \,\bz)= \bz^T \nabla(\bz\cdot \bw).
%\]
%Assume that there is $x \in \Omega$ such that $\bw(x)\cdot \nabla |\bz(x)|^2 \neq 0$. Then $\nabla(\bz(x)\cdot \bw(x)) \neq 0$ and thus $ \bw \notin V_{\bz}^0$. For the other implication, assume $\bw \in V_{\bz}^0$, hence $\nabla(\bz\cdot \bw)=0$ on $\Omega$ and thus $\bw\cdot \nabla |\bz|^2 = 0$ on $\Omega$.
%\end{proof}
\ \\
%From \eqref{implic2} it follows that for $\bw \in (V_{\hat \bz}^0 \cap K)$, $\bw \neq 0$, the length of the flow field $\bz$ must be constant along the flow lines of the rigid motion $\bw$. This already indicates that $\bz$ is ``rather special''.
%
\subsection{Examples} \label{sect2}
We discuss two specific examples  with known solutions. 
The  first one is  the example explained in Remark~\ref{ex1}. For a fixed $t \in [0,\tfrac14]$ we take $\phi(\cdot,t)$ as in Remark~\ref{ex1}. The transport of this $\phi$ corresponds to the rigid motion $\bu(\bx)=(1~0)^T$, $\bx \in \R^2$.  We take
\[
   \bz = z \hat \bz=(\bu \cdot \bn)  \bn = - \frac{\partial \phi}{\partial t} \frac{\nabla \phi}{\|\nabla \phi\|^2}.
\]
This flow field in normal direction is available from the level set function  information. 
For this case we have $\cK_{\hat \bz}=\{\mathbf{0}\}$ and the solution of the minimization problem \eqref{problem1} is  the rigid motion  $\bu$, cf.~Remark~\ref{remrigid}, with  energy $a(\bu,\bu)=0$. Thus we recover, in this example, the rigid motion from  the level set function information only. 
\ 
\\[1ex]

We introduce another class of problems for which an analytical solution is known. In this example we have $\dim(\cK_{\hat \bz})=1$, cf. also Remark~\ref{remRes} below. 
\begin{lemma} \label{lemspecial}
We take $d=2$ and define ${\hat \bz}:=\begin{pmatrix} c \\ s\end{pmatrix}$, $\bgp:=\begin{pmatrix} -s \\ c\end{pmatrix}$, with $c^2+s^2=1$. Take  
%$c_0 \in \R$ and 
a smooth function $F:\R \to \R$ with
\begin{equation} \label{Gauge}
  \int_{\Omega}F\big({\hat \bz} \cdot\begin{pmatrix} x \\ y \end{pmatrix}\big) \, d \bx=0. 
\end{equation}
Consider \eqref{problem1} with 
\begin{equation} \label{defz} \bz(x,y):=\left(\bgp \cdot\begin{pmatrix} x \\ y \end{pmatrix} F'\big({\hat \bz} \cdot\begin{pmatrix} x \\ y \end{pmatrix}\big) \right) {\hat \bz}.
\end{equation}
Then
\begin{equation} \label{defv}
  \bv(x,y)= - F\big({\hat \bz} \cdot\begin{pmatrix} x \\ y \end{pmatrix}\big) \bgp
\end{equation}
is the unique solution of \eqref{problem1}.
\end{lemma}
\begin{proof} Given in the Appendix.
\end{proof}
\ \\[1ex]
\begin{remark} \label{remRes}\rm The Gauge condition \eqref{Gauge} is introduced to guarantee that $\bv$ as in \eqref{defv} satisfies the orthogonality condition \eqref{orthocond}, i.e., $\bv \perp_{H^1} \cK_{\hat \bz}= {\rm span}( \hat \bz_\perp)$ holds. In the example of Lemma~\ref{lemspecial} the given velocity field $\bz = z \hat \bz$ has a spatially constant direction $\hat \bz = (c ~s)^T$ and has a varying length $z(x,y)$. One aims to correct with a $\bv$ that is pointwise orthogonal to $\hat \bz$, i.e. $\bv$ has the constant direction $\hat \bz_{\perp}$, and has varying length. The optimal (w.r.t. strain energy minimization) correction is given by $\bv$ in  \eqref{defv}.  As a special case one can consider a domain $\Omega$ that is symmetric with respect to the $y$-axis, i.e., $(x,y) \in \Omega$ iff $(-x,y) \in \Omega$, $c=1$, $s=0$ and the function $F(\alpha)=\alpha$. This then satisfies \eqref{Gauge} and we obtain $\bz(x,y)=(y~0)^T$ with corresponding $\bv(x,y)= (0 ~-x)^T$. We thus obtain a velocity field $\bu=\bz+\bv=(y~-x)^T$ that is a rigid motion, hence has a strain energy $a(\bu,\bu)=0$. In general, however, for other choices of $F$ we have $a(\bu,\bu) >0$. 
\end{remark}
%The role of the constant $c_0$ in \eqref{defz} is discussed in Remark~\ref{remregul}.

\section{Reformulation as saddle point problem} \label{sect3}
To obtain a feasible solution method we introduce a Lagrange multiplier to eliminate the constraint $\bv\cdot \hat \bz = 0$. 
 We define $b: \, H^1(\Omega) \times \bH^1(\Omega) \to \R$, $b(\lambda, \bv):=(\lambda, \bv \cdot \hat \bz )_1$, where $(\cdot,\cdot)_1$ is the standard scalar product on $H^1(\Omega)$, i.e., $(v,w)_1 =\int_\Omega \nabla v \cdot \nabla w + v w\, d\bx$. For the kernel of $b(\cdot,\cdot)$   we have  
\[
  \ker (b) = \{\, \bv \in \bH^1(\Omega)~|~b(\mu,\bv)=0 \quad \forall ~\mu \in H^1(\Omega)\,\}=V_{\hat \bz}^0.
\]
We introduce
 \begin{equation} \label{defH} \bH_\ast^1(\Omega):=  \bH^1(\Omega) \setminus \cK_{\hat \bz} =\{\, \bv \in \bH^1(\Omega)~|~(\bv,\bw)_1=0 ~\text{for all}~ \bw \in \cK_{\hat \bz}\,\}.
\end{equation}
Note that in the generic case we have $\cK_{\hat \bz}=\{\mathbf{0}\}$ and then $\bH_\ast^1(\Omega)=\bH^1(\Omega)$ holds.
There exists a constant $C_d$, depending only on the dimension $d$, such that the following holds:
\begin{equation} \label{Cd}
  \|\bv \cdot \bw \|_1 \leq C_d \|\bv \|_1 \|\bw\|_{H^{1, \infty}} \quad \text{for all}~\bv \in \bH^1(\Omega),~\bw \in \bH^{1,\infty}(\Omega).
\end{equation}
This constant $C_d$ will be used in continuity and inf-sup estimates below. 
%Here $(\cdot,\cdot)_1$ is the scalar product on $\bH^1(\Omega)$. We use the factorization of the space $\cK_{\hat \bz}$ with respect to this $H^1$-scalar product since this is convenient in the analysis. 
Below, the orthogonal projection in $\bH^1(\Omega)$ w.r.t. $(\cdot,\cdot)_1$ on a subspace $W$ is denoted by $Q_{W}$. 
\begin{lemma} \label{lemmahulp}
 We consider the saddle point problem: determine $(\bw, \lambda)  \in  \bH_\ast^1(\Omega) \times H^1(\Omega)$ such that
\begin{equation} \label{problem2a}
\begin{split}
  a(\bw,\bv)+ b(\lambda, \bv) &= -a(\bz,\bv) \quad \text{for all}~~\bv \in  \bH_\ast^1(\Omega) \\
   b(\mu, \bw) &= 0 \quad \text{for all}~~\mu \in  H^1(\Omega).
\end{split}
\end{equation}
This problem is well-posed.
\end{lemma}
\begin{proof}
Note that $a(\cdot,\cdot)$ is continuous on $\bH_\ast^1(\Omega)$ and elliptic on $\bH_\ast^1(\Omega)\cap \ker (b)=\ker (b) \setminus (\ker (b)\cap K)$. Using Assumption~\ref{ass1} we get 
\begin{equation} \label{contA}
  |b(\lambda,\bv)| =|(\lambda, \bv \cdot \hat \bz)_1| \leq c_{\rm cont} \|\lambda\|_1 \|\bv\|_1,  \quad c_{\rm cont}:=C_d \|\hat \bz\|_{H^{1,\infty}}.
\end{equation}
For the right-hand side in \eqref{problem2a} we have that  $\bv \to a(\bz,\bv)$ is continuous on $\bH^1(\Omega)$.
It remains to check the inf-sup property for $b(\cdot,\cdot)$. Take $\lambda \in H^1(\Omega)$ and define $\bv:=\hat \bz \lambda - Q_{\cK_{\hat \bz}}(\hat \bz \lambda)\in  \bH_\ast^1(\Omega)$. For this choice we have $\|\bv\|_1 \leq \|\hat \bz \lambda\|_1 \leq C_d \|\hat \bz\|_{H^{1,\infty}} \|\lambda\|_1$. Note that $\cK_{\hat \bz} \subset \ker (b)$ and thus $b(\lambda,Q_{\cK_{\hat \bz}} \bw)=0$ for all $\bw \in \bH^1(\Omega)$.  Using this we get 
\begin{equation} \label{infsup}
 \frac{b(\lambda,\bv)}{\|\bv\|_1} =\frac{(\lambda,\lambda)_1}{\|\bv\|_1} \geq c_{\rm infsup}\|\lambda\|_1, \quad c_{\rm infsup}:=c_{\rm cont}^{-1}.
\end{equation}
and it follows that  $b(\cdot,\cdot)$ has the inf-sup property on $H^1(\Omega) \times \bH_\ast^1(\Omega)$. Hence, the problem \eqref{problem2a} has a unique solution. 
\end{proof}
\ \\[1ex]
A saddle point reformulation of the problem \eqref{problem1} is as follows: determine $(\bu,\lambda) \in \bH^1(\Omega) \times H^1(\Omega)$ with $Q_{\cK_{\hat \bz}}\bu=Q_{\cK_{\hat \bz}}\bz$  such that
\begin{equation} \label{problem3}
\begin{split}
  a(\bu,\bv)+ b(\lambda, \bv) &= 0 \quad \text{for all}~~\bv \in  \bH^1(\Omega) \\
   b(\mu, \bu) &= b(\mu, \bz)=(\mu,z)_1 \quad \text{for all}~~\mu \in  H^1(\Omega).
\end{split}
\end{equation}
For testing we use $\bv \in  \bH^1(\Omega)$ instead of $\bv \in \bH_\ast^1(\Omega)= \bH^1(\Omega) \setminus \cK_{\hat \bz}$. This is allowed because for $\bv\in \cK_{\hat \bz} = \ker(b) \cap K$ we have $a(\bu,\bv)=0$ and $b(\lambda,\bv)=0$. 
From Lemma~\ref{lemmahulp} and a shift argument $\bu=\bz+\bw$ it follows that \eqref{problem3} is well-posed. 
 Let $(\bu,\lambda)$ be the solution of \eqref{problem3}. Testing with $\bv \in V_{\hat \bz}^0= \ker(b)$ we obtain $a(\bu,\bv)=0$ for all $\bv \in V_{\hat \bz}^0$. From this it follows that $\bu$ is the unique solution of \eqref{problem1}. We conclude that \eqref{problem3} is a \emph{well-posed saddle point reformulation} of the constrained minimization problem \eqref{problem1}. This saddle point formulation will be the basis for the finite element discretization that we use to solve the minimization problem approximately, cf. Section~\ref{secdiscr} below. We note that the equivalence between the formulations \eqref{problem2a} and \eqref{problem3} via the shift $\bu=\bz+\bw$, in general does \emph{not} hold for the corresponding discrete versions of these problems. The reason for this is that a shift property $\bu_h=\bz+\bw_h$ with finite element functions $\bu_h$, $\bw_h$, implies that $\bz$ is a finite element function, which in general is not the case.   
 
 \begin{remark} \rm 
 We  comment on the choice of the space $\bH_\ast^1(\Omega)= \bH^1(\Omega) \setminus \cK_{\hat \bz}$ in \eqref{defH}.
As shown above, with this choice the Lagrange multiplier formulation in \eqref{problem3} is equivalent to \eqref{problem1} in the sense that the $\bu$-solutions of both problems are the same. Alternatively one might consider a Lagrange multiplier formulation with the (simpler) choice $\tilde \bH_\ast^1(\Omega):= \bH^1(\Omega) \setminus K \subset \bH_\ast^1(\Omega)$. Note that the bilinear form $a(\cdot,\cdot)$ is elliptic on this space but not on  $\bH_\ast^1(\Omega)$. With the choice $\tilde \bH_\ast^1(\Omega)$, however, in general  the Lagrange multiplier formulation and \eqref{problem1} do not have the same solution. As an example, consider the problem presented in Remark~\ref{ex1}, consisting of an ellipse that is transported by the rigid motion $\bu= (1~0)^T$. In Section~\ref{sect2} it is shown that we obtain this $\bu$ as minimizer of the constrained minimization problem \eqref{problem1} and thus also as solution of the equivalent saddle point problem \eqref{problem3}. If, however, we use $\tilde \bH_\ast^1(\Omega)$ then, due to $\bu \notin\tilde \bH_\ast^1(\Omega)$, we will not obtain this $\bu$ as solution of the saddle point problem with this alternative space.  
%Since our aim is  to minimize the strain energy $a(\cdot,\cdot)$ and $a(\bu^\ast, \bu^\ast)=0$ holds, we want to have a Lagrange multiplier formulation that yields the solution $\bu^\ast$. This is indeed the case if we use $\bH_\ast^1(\Omega)$,  but does not hold for the choice $\tilde \bH_\ast^1(\Omega)$. 
 \end{remark}

 %Testing with $\bv \in \bH^1(\Omega)$ in \eqref{problem3} is convenient from an implementation point of view. In an implementation one would seek for an arbitrary  solution $\bu \in \bH^1(\Omega)$ of \eqref{problem3} and then, in case $\dim (\cK_{\hat \bz}) \geq 1$ holds, enforce $P_{\cK_{\hat \bz}}\bu=P_{\cK_{\hat \bz}}\bz$ with postprocessing this solution $\tilde \bu = \bu - P_{\cK_{\hat \bz}}(\bu-\bz)$.  For this one needs a basis of $\cK_{\hat \bz}$. 

\begin{remark}   \label{remLa}
 \rm We reconsider the special class of problems presented in Lemma~\ref{lemspecial}. We briefly discuss the corresponding saddle point formulation and in particular investigate the corresponding Lagrange multiplier solution $\lambda$. In this example we have $\cK_{\hat \bz}={\rm span}(\hat \bz_\perp)$. Since $\hat \bz_\perp$ is a constant vector we have
 \[ Q_{\cK_{\hat \bz}} \bu = \frac{(\bu,\bgp)_{1}}{(\bgp,\bgp)_{1}}\bgp=\frac{(\bu,\bgp)_{L^2}}{(\bgp,\bgp)_{L^2}}\bgp =\frac{(\bu,\bgp)_{L^2}}{|\Omega|}\bgp. 
 \]
  Let $\bz= z_1 \hat \bz$ and $\bv=v_2 \bgp$ be as in \eqref{defz} and \eqref{defv}, respectively. %We assume $c_0$ sufficiently large such that $z_1 > 0$ on $\Omega$ holds. We assume that $z_1$ and $v_2$ are smooth (which is guarenteed if $F\in C^\infty(\R)$). 
 Note that 
 $b(\lambda, \bv)=(\lambda, \bv\cdot \hat \bz)_1$. We test in \eqref{problem3} with $\bv=w_1 {\hat \bz}$, $w_1 \in H^1(\Omega)$. This yields
 \begin{equation} \begin{split}
  b(\lambda, \bv) & =(\lambda,w_1)_1= -a(\bu,w_1{\hat \bz}) = -a(z_1{\hat \bz} +v_2 \bgp, w_1{\hat \bz}) \\
    & = -2 \int_\Omega \nabla z_1 \cdot \nabla w_1 + (\nabla z_1 \cdot \hat \bz)(\nabla w_1 \cdot {\hat \bz}) +(\nabla v_2\cdot {\hat \bz})(\nabla w_1\cdot \bgp) \, d\bx 
    \end{split}
 \end{equation}
for all $w_1 \in H^1(\Omega)$. We introduce notation for the smooth vectors $\bF_\bz:=2 \nabla z_1$, $\bF_\bv:=2 \nabla v_2$. Hence, we conclude that the Lagrange multiplier is the unique solution of the elliptic problem:
\begin{equation} 
   (\lambda,w_1)_1=- \int_\Omega \bF_\bz \cdot \nabla w_1 + (\bF_\bz\cdot {\hat \bz})(\nabla w_1 \cdot {\hat \bz})+(\bF_\bv\cdot {\hat \bz})(\nabla w_1\cdot \bgp) \, d\bx~\forall~w_1 \in H^1(\Omega).
\end{equation}
In strong formulation this problem takes the form of a Neumann problem:
\begin{equation} \label{Lagrangeeq}
 \begin{split}
  - \Delta \lambda + \lambda & = \Div \bF_\bz +{\hat \bz} \cdot \nabla (\bF_\bz \cdot {\hat \bz}) + \bgp \cdot \nabla(\bF_\bv \cdot {\hat \bz})\quad \text{in}~\Omega \\
   \bn \cdot \nabla \lambda &= -\bn \cdot \bF_\bz - ({\hat \bz} \cdot \bn) (\bF_\bz \cdot {\hat \bz})- (\bgp \cdot \bn)(\bF_\bv \cdot {\hat \bz})\quad \text{on}~ \partial\Omega, 
 \end{split}
\end{equation}
where $\bn$ denotes the unit normal on $\partial \Omega$.  The source terms on the right hand sides in \eqref{Lagrangeeq} are smooth if the function $F$ in Lemma~\ref{lemspecial}  is smooth. Nevertheless the solution $\lambda$ may have large gradients due to boundary effects, e.g., if the boundary has a reentrant corner. We will illustrate this in a numerical experiment in Section~\ref{sectcorner}. 

\end{remark}

\begin{lemma} \label{lemma3}
 Let  $(\bu,\lambda)$ be the solution of problem \eqref{problem3}.  The following estimates hold with $c_{\rm cont}$ defined in \eqref{contA}:
 \begin{align}
   \frac{a(\bu,\bu)}{\|\bu\|_1} & \leq c_{\rm cont} \|\lambda\|_1  \label{estu}\\
   \|\lambda\|_1 & \leq  2c_{\rm cont} a(\bu,\bu)^\frac12. \label{estlambda}
 \end{align}
\end{lemma}
\begin{proof}
 The estimate in \eqref{estu} follows from
 \[
   a(\bu,\bu)= -b(\lambda,\bu) \leq c_{\rm cont} \|\lambda\|_1 \|\bu\|_1,
 \]
cf. \eqref{contA}. The result in \eqref{estlambda} follows  from $a(\bv,\bv) \leq 4 \int_{\Omega} \nabla \bv:\nabla \bv\, d\bx \leq 4 \|\bv\|_1^2$ and \eqref{infsup}:
\[  \begin{split} 
 \|\lambda\|_1  & \leq c_{\rm cont} \sup_{\bv \in \bH_\ast^1(\Omega)} \frac{b(\lambda,\bv)}{\|\bv\|_1} = c_{\rm cont} \sup_{\bv \in \bH_\ast^1(\Omega)} \frac{a(\bu,\bv)}{\|\bv\|_1} \\
 & \leq c_{\rm cont} \sup_{\bv \in \bH_\ast^1(\Omega)} \frac{a(\bu,\bu)^\frac12 a(\bv,\bv)^\frac12 }{\|\bv\|_1} \leq 2 c_{\rm cont} a(\bu,\bu)^\frac12.
\end{split} \]
\end{proof}
\ \\
From this result we see that $\|\lambda\|_1^2$ is a measure for $a(\bu,\bu)$, i.e., the value of the strain energy functional at the minimizer $\bu$. 

%\begin{remark} \label{remredundant}
% \rm The approach above can be modified if the given velocity $\bz$ vanishes on a part of $\Omega$ with strictly positive measure. In such a case one can proceed as follows. We keep the assumptions \eqref{regassumptions}. Define  $\Omega_s := \overline{{\rm supp}(|\bz|)}\subset \Omega$ and use as Lagrange multiplier space $H^1(\Omega_s)$. For the corresponding bilinear form we use $b(\lambda, \bv)  =\int_{\Omega_s} \nabla \lambda \cdot \nabla (\bv \cdot \hat \bz) + \lambda (\bv \cdot \hat \bz) \, dx$.
%\end{remark}
\section{Discretization of the saddle point problem}  \label{secdiscr}
We assume that $\Omega$ is polygonal. Let $\{\T_h\}_{h>0}$ be a family of shape regular (not necessarily quasi-uniform)  simplicial triangulations of $\Omega$. The space of continuous piecewise polynomial functions of degree $k \geq 1$ on $\cT_h$ is denoted by $V_h$ and $\bV_h := V_h^d$ is the corresponding space of vector values finite element functions.
In the discretization we use the bilinear form $b(\cdot,\cdot)$ restricted to the finite element spaces  with kernel denoted by
\[
\ker_h(b):=\{\, \bv_h \in \bV_h~|~b(\lambda_h,\bv_h)=0 \quad \text{for all}~\lambda_h \in V_h\,\}.
\]
Recall that  $\K_{\hat \bz}= \ker(b)\cap K$. A discrete analogue of this space is 
\[
 \K_{\hat \bz,h}:= \ker_h(b)\cap K.
\]
We define 
$\bV_{h,\ast}:= \bV_h \setminus \cK_{\hat \bz,h}=\{\, \bv_h \in \bV_h~|~Q_{\Kzh}\bv_h=0\,\}$. The pair $\bV_{h,\ast} \times V_h$ is suitable for a  Galerkin discretization of \eqref{problem2a}, since $V_h \subset H^1(\Omega)$ and, due to $\cK_{\hat \bz} \subset \cK_{\hat \bz,h}$, also $\bV_{h,\ast} \subset \bH_{\ast}^1(\Omega)$. This pair is used in the analysis in Section~\ref{secerroranalysis}.
Note that $ \bV_{h,\ast} \cap \ker_h(b)= \ker_h(b) \setminus (\ker_h(b) \cap K)$ holds. 

We introduce a discretization of problem \eqref{problem3}: determine $(\bu_h,\lambda_h) \in \bV_h \times V_h$ with $Q_{\cK_{\hat \bz,h}} \bu_h= Q_{\cK_{\hat \bz,h}} \bz$ such that
\begin{equation} \label{problem3discr}
\begin{split}
  a(\bu_h,\bv_h)+ b(\lambda_h, \bv_h) &= 0 \quad \text{for all}~~\bv_h \in \bV_h\\
   b(\mu_h, \bu_h) &= b(\mu_h, \bz)=(\mu_h,z)_1 \quad \text{for all}~~\mu_h \in  V_h.
\end{split}
\end{equation}
Well-posedness of this problem is analyzed below. The analysis uses standard techniques for finite element error analysis of saddle point problems, but we need some additional technical results to also cover the case $\dim(\K_{\hat \bz}) \geq 1$. If we would restrict to the case $\K_{\hat \bz}=\{\mathbf{0}\}$ the results on well-posedness of the discrete problem presented below and the error bound in Theorem~\ref{Thmmain} can be obtained with a simpler analysis.

\subsection{Discretization error analysis} \label{secerroranalysis}
In this section we assume additional regularity of $\hat \bz$, namely $\hat \bz \in \bH^{k+1,\infty}(\Omega)$. In the error analysis below we use $c$ to denote a generic constant that is independent of the mesh size parameter $h$ and whose value can change from one instance to another.
We start with an elementary lemma. We use the orthogonal projections $Q_{V_h}: H^1(\Omega) \to V_h$, $Q_{\bV_h}: \bH^1(\Omega) \to \bV_h$, with respect to $(\cdot,\cdot)_1$. Note that $\bv_h \in \ker_h(b)$ iff $Q_{V_h}(\bv_h \cdot \hat \bz)=0$ holds.
\begin{lemma} \label{lemdiscr1}
 The following estimates hold with constants $c$ independent of $h$:
 \begin{align}
  \|\bv_h \cdot \hatbz - Q_{V_h}(\bv_h \cdot \hatbz)\|_1 & \leq c h \|\bv_h\|_1 \|\hatbz\|_{H^{k+1,\infty}(\Omega)} \quad \text{for all}~\bv_h \in \bV_h,\label{res11}
  \\
   \|\lambda_h \hatbz - Q_{\bV_h}(\lambda_h \hatbz)\|_1 & \leq c h \|\lambda_h \|_1 \|\hatbz\|_{H^{k+1,\infty}(\Omega)} \quad \text{for all}~\lambda_h \in V_h,\label{res11a}\\
 \|\bv\cdot \hatbz - Q_{V_h}(\bv \cdot \hatbz)\|_1 & \leq c h^k \|\bv\|_1 \|\hatbz\|_{H^{k+1,\infty}(\Omega)} \quad \text{for all}~\bv \in K. \label{res12}
 \end{align}
\end{lemma}
\begin{proof}
 Take $\bv_h \in \bV_h$. Let $I_{V_h}$ be the nodal interpolation operator in $V_h$. On an element $T \in \T_h$ we get, using a finite element inverse inequality and $|\bv_h|_{k+1,T}=0$:
 \begin{align*}
  & \|\bv_h \cdot \hatbz - I_{V_h}(\bv_h \cdot \hatbz)\|_{H^1(T)}  \leq c\, h_T^k |\bv_h \cdot \hatbz|_{k+1,T} \leq c \, h_T^k \sum_{\ell=0}^{k+1} |\bv_h|_{\ell,T}|\hatbz|_{k+1-\ell,\infty,T} \\
  & \qquad  \leq c \,  h_T^k \sum_{\ell=0}^{k} |\bv_h|_{\ell,T}|\hatbz|_{k+1-\ell,\infty,T} \\ & \qquad \leq c\,  h_T \|\bv_h\|_{H^1(T)} \sum_{\ell=0}^{k}|\hatbz|_{k+1-\ell,\infty,T} 
  \leq c\,  h_T \|\bv_h\|_{H^1(T)} \|\hatbz\|_{H^{k+1,\infty}(T)}.
 \end{align*}
Squaring this,  summing over $T \in \T_h$ and using $\|\bv_h \cdot \hatbz - Q_{V_h}(\bv_h \cdot \hatbz)\|_1 \leq \|\bv_h \cdot \hatbz - I_{V_h}(\bv_h \cdot \hatbz)\|_1 $ we obtain the result \eqref{res11}. With the same arguments we obtain the result in \eqref{res11a}.
\\
Note that $\bv \in K$ is a constant or linear vector function and thus $|\bv|_{\ell,T}=0$ for $\ell \geq 2$. Using this we
get 
\begin{align*}
  & \|\bv \cdot \hatbz - I_{V_h}(\bv \cdot \hatbz)\|_{H^1(T)}\leq c\,  h_T^k |\bv \cdot \hatbz|_{k+1,T}  \leq c\,  h_T^k \sum_{\ell=0}^{k+1} |\bv|_{\ell,T}|\hatbz|_{k+1-\ell,\infty,T} \\
 & \qquad \leq c\, h_T^k \sum_{\ell=0}^{1} |\bv|_{\ell,T}|\hatbz|_{k+1-\ell,\infty,T} \leq c\, h_T^k \|\bv\|_{H^1(T)} \|\hatbz\|_{H^{k+1,\infty}(T)}.
\end{align*}
With the same arguments as above we obtain the estimate \eqref{res12}.
\end{proof}
\ \\[1ex]
We now quantify the difference between the space $\cK_{\hat \bz}$ and $\cK_{\hat \bz,h}$. Both $\K_{\hat \bz}$ and $\K_{\hat \bz,h}$ are linear subspaces of $K$ and $\K_{\hat \bz} \subset \K_{\hat \bz,h}$ holds (this follows using $K \subset \bV_h$). For distance between  subspaces we use the usual angle notion and define ${\rm dist}(\K_{\hat \bz,h},\K_{\hat \bz}):= \sup_{\bw \in \K_{\hat \bz,h}, \|\bw\|_1=1} \inf_{\bv \in \K_{\hat \bz}} \|\bw- \bv\|_1$. Note that ${\rm dim}(\K_{\hat \bz,h}) > {\rm dim}(\K_{\hat \bz})$ implies ${\rm dist}(\K_{\hat \bz,h},\K_{\hat \bz})=1$. 
\begin{lemma} \label{lemdiscr2}
 The following holds with a constant $c$ independent of $h$:
 \begin{equation} \label{dist}
  {\rm dist}(\K_{\hat \bz,h},\K_{\hat \bz}) \leq c h^k.
 \end{equation}
This implies that for $h$ sufficiently small $\K_{\hat \bz}=\K_{\hat \bz,h}$ holds.
\end{lemma}
\begin{proof}
We use the orthogonal decomposition $K= \cK_{\hat \bz} \oplus \cK_{\hat \bz}^\perp$ with  $\cK_{\hat \bz}^\perp:= \{\, \bv \in K~|~(\bv,\bw)_1=0 \quad \text{for all}~\bw \in \K_{\hat \bz}\,\}$.  Define the bounded linear operator $B: \, K \to H^{-1}(\Omega)$ by $(B \bw)(\mu):= b(\mu, \bw)=(\mu, \bw \cdot \hat \bz)_1$. Hence, $\ker (B)= \cK_{\hat \bz}$ and thus $B$ is injective on the finite dimensional space $\cK_{\hat \bz}^\perp$. This implies that there exists a constant $c_0$ such that $\|\bv\|_1 \leq c_0 \|B\bv\|_{H^{-1}}$ holds for all $\bv \in \cK_{\hat \bz}^\perp$. Take $\bw \in \K_{\hat \bz,h} \subset K$ with $\|\bw\|_1=1$ and decompose as $\bw=\bw_1 +\bw_2$, $\bw_1\in \cK_{\hat \bz}$, $\bw_2 \in \cK_{\hat \bz}^\perp$. 
 Recall  that $\bv \in \K_{\hat \bz,h}$ implies that  $Q_{V_h}(\bv \cdot \hatbz)=0$ holds. Using this and \eqref{res12} we obtain
 \begin{align*}
   (B \bw)(\mu) & =(\mu, \bw \cdot \hat \bz)_1= \big(\mu, \bw \cdot \hatbz- Q_{V_h}(\bw \cdot \hatbz)\big)_1 \leq c h^k \|\mu\|_1.
\end{align*}
Thus we get the estimate $\|B\bw \|_{H^{-1}} \leq c h^k$. This yields
\[
 \inf_{\bv \in \K_{\hat \bz}} \|\bw- \bv\|_1 = \|\bw_2\|_1 \leq c_0 \|B\bw_2\|_{H^{-1}} = c_0\|B\bw\|_{H^{-1}} \leq c h^k,
\]
which completes the proof.
\end{proof}
\ \\
In the analysis below we need that the spaces $\ker (b) \setminus \cK_{\hat \bz}$ and its discrete analogue $\ker_h(b) \setminus \cK_{\hat \bz,h}$ are 
``close'' (for $h$ sufficiently small). The following lemma provides a useful result. 
%The result above yields that for $h$ sufficiently small we have $\ker_h (b) \setminus \cK_{\hat \bz,h} =\ker_h(b) \setminus \cK_{\hat \bz}$. We now compare $\ker (b) \setminus \cK_{\hat \bz}$ and $\ker_h(b) \setminus \cK_{\hat \bz}$. Here the subspace factorization $\setminus$ is done using the $H^1$-orthogonal projection.
\begin{lemma} \label{lemcomp}
 %Assume $h$ sufficiently small such that $\cK_{\hat \bz}=\cK_{\hat \bz,h}$ holds.
 For arbitrary $\bv_h \in \ker_h(b) \setminus \cK_{\hat \bz,h}$ there exists $\bv \in \ker(b) \setminus \cK_{\hat \bz}$ such that
 \[
   \|\bv_h - \bv\|_1 \leq c h \|\bv_h\|_1
 \]
holds with a constant $c$ independent of $h$, $\bv_h$. 
\end{lemma}
\begin{proof}
%Note that $\cK_{\hat \bz}=\cK_{\hat \bz,h}$ is a finite dimensional subspace of both $\ker (b)$ and $\ker_h(b)$. 
Take $\bv_h \in \ker_h(b) \setminus \cK_{\hat \bz,h}$, i.e. $Q_{V_h}(\bv_h \cdot \hat \bz)=0$ and $Q_{\cK_{\hat \bz,h}}\bv_h =0$. Due to $\cK_{\hat \bz} \subset \cK_{\hat \bz,h}$ we also have $Q_{\cK_{\hat \bz}}\bv_h =0$. Define $\tilde \bv_h=\bv_h - (\bv_h \cdot \hat \bz)\hat \bz \in \ker(b)$ and $\bv=\tilde \bv_h - Q_{\cK_{\hat \bz}}\tilde \bv_h \in \ker (b) \setminus \cK_{\hat \bz}$. We then obtain
 \begin{align*}
  \|\bv_h - \bv \|_1 & =\|(I-Q_{\cK_{\hat \bz}})(\bv_h-\tilde\bv_h)\|_1 \leq \|(\bv_h \cdot \hat \bz) \hat \bz\|_1 \\
  &  \leq c \|\bv_h \cdot \hat \bz\|_1\|\hat \bz\|_{H^{1,\infty}} \leq c\|\bv_h \cdot \hat \bz -Q_{V_h}(\bv_h\cdot \hat \bz)\|_1 \leq c h \|\bv_h\|_1,
 \end{align*}
where in the last inequality we used \eqref{res11}.
\end{proof}
\ \\
We now derive discrete ellipticity and inf-sup properties.
\begin{lemma} \label{lemellip}
 There exists $\gamma >0$ independent of $h$ such that 
 \begin{equation}\label{ellip}
  a(\bv_h,\bv_h) \geq \gamma \|\bv_h\|_1^2 \quad \text{for all}~ \bv_h \in  \bV_{h,\ast} \cap \ker_h(b). 
 \end{equation}
\end{lemma}
\begin{proof}
 Note that $\bV_{h,\ast} \cap \ker_h(b)= \ker_h(b) \setminus (\ker_h(b) \cap K)=\ker_h(b) \setminus \cK_{\hat \bz,h}$ and $K$ is the kernel of $a(\cdot,\cdot)$. Hence, there exists a possibly $h$-dependent $\gamma_h >0$ such that \eqref{ellip} holds with $\gamma$ replaced by  $\gamma_h$. It remains to show that $\gamma_h$ is uniformly in $h$ bounded away from $0$. 
 %Note that $\ker_h(b) \setminus (\ker_h(b) \cap K)= \ker_h(b) \setminus \cK_{\hat \bz,h}$. 
 We use that $a(\cdot,\cdot)$ is elliptic on $\ker (b)\setminus \cK_{\hat \bz}$, i.e. for suitable $\hat \gamma >0$ we have $a(\bv,\bv) \geq \hat \gamma \|\bv\|_1^2$ for all $\bv \in \ker (b) \setminus \cK_{\hat \bz}$. The seminorm corresponding to $a(\cdot,\cdot)$ is denoted by $\|\cdot\|_a$. 
 %Assume $h$ sufficiently small such that $\cK_{\hat \bz}=\cK_{\hat \bz,h}$ holds. 
 Take $\bv_h \in \ker_h(b) \setminus \cK_{\hat \bz,h}$ and a  corresponding $\bv\in \ker (b) \setminus \cK_{\hat \bz}$ as in Lemma~\ref{lemcomp}. Using $\|\bv_h -\bv\|_a \leq 2 \|\bv_h-\bv\|_1$ we obtain, with suitable constants $c>0$,
 \begin{align*}
  a(\bv_h,\bv_h)^\frac12 & =\|\bv_h\|_a  \geq \|\bv\|_a - \|\bv_h-\bv\|_a \geq \hat \gamma^\frac12 \|\bv\|_1 - 2 \|\bv_h-\bv\|_1 \\
   & \geq \hat \gamma^\frac12 \|\bv_h\|_1 - c \|\bv_h -\bv\|_1 \geq \hat \gamma^\frac12 \|\bv_h\|_1(1- ch),
 \end{align*}
 and combining this with $a(\bv_h,\bv_h) \geq \gamma_h \|\bv_h\|_1^2$ completes the proof. 
\end{proof}

\begin{lemma}
There is  a constant $\beta >0$, independent of $h$ such that for $h$ sufficiently small the following holds:
\begin{equation} \label{discrinfsup}
  \sup_{\bv_h \in \bV_{h,\ast}} \frac{b(\lambda_h, \bv_h)}{\|\bv_h\|_1} \geq \beta \|\lambda_h\|_1 \quad \text{for all}~\lambda_h \in V_h.
\end{equation}
\end{lemma}
\begin{proof}
Take $\lambda_h \in V_h$. Define $\bv_h := Q_{\bV_h}(\lambda_h \hat \bz)- Q_{\cK_{\hat \bz,h}} Q_{\bV_h}(\lambda_h \hat \bz) \in \bV_{h,\ast}$. Note that $\|\bv_h\|_1 \leq c \|\lambda_h\|_1$ holds, cf. \eqref{Cd}. Note that $\cK_{\hat \bz,h} \subset \ker_h(b)$ and thus $b(\lambda_h, Q_{\cK_{\hat \bz,h}} \bw_h)=0$ holds for all $\bw_h \in \bV_h$. Using this and \eqref{res11a} we obtain, with suitable constants $c>0$, 
\begin{align*}
 b(\lambda_h, \bv_h) & = b(\lambda_h, Q_{\bV_h}(\lambda_h \hat \bz)) \\
                    &= b(\lambda_h, \lambda_h \hat \bz) - b(\lambda_h, \lambda_h \hat \bz-Q_{\bV_h}(\lambda_h \hat \bz)) \\
                    & \geq \|\lambda_h\|_1^2 - c \|\lambda_h\|_1 \|\lambda_h \hat \bz-Q_{\bV_h}(\lambda_h \hat \bz)\|_1 \\
                    & \geq (1-ch)\|\lambda_h\|_1^2.
\end{align*}
Combining these results completes the proof.
\end{proof}
 \begin{assumption}
  In the remainder we assume that $h$ is sufficiently small, such that the discrete inf-sup property \eqref{discrinfsup} holds. 
 \end{assumption}
\ \\[1ex]
Using this we derive a well-posedness result for the discrete problem \eqref{problem3discr}.
\begin{lemma} \label{Wellposed2}
The problem \eqref{problem3discr} has a unique solution.
\end{lemma}
\begin{proof}
First consider the saddle point problem: determine $(\hat \bu_h,\lambda_h) \in \bV_{h,\ast} \times V_h$ such that
\begin{equation} \label{problem2discr}
\begin{split}
  a(\hat \bu_h,\bv_h)+ b(\lambda_h, \bv_h) &= 0 \quad \text{for all}~~\bv_h \in \bV_{h,\ast}\\
   b(\mu_h, \hat \bu_h) &= b(\mu_h,\bz) \quad \text{for all}~~\mu_h \in  V_h.
\end{split}
\end{equation}
Using the discrete ellipticity and inf-sup properties, an application of standard saddle point theory \cite{GR,Ern04} yields that this problem has a unique solution $(\hat \bu_h,\lambda_h)$. Now define $\bu_h:=\hat \bu_h + Q_{\Kzh}\bz$. From $Q_{\Kzh}\hat \bu_h=0$ we obtain $Q_{\Kzh}\bu_h=Q_{\Kzh}\bz$.
Note that $a(Q_{\Kzh}\bz,\bv_h)=0$ for all $\bv_h \in \bV_h$ and $b(\mu_h,Q_{\Kzh}\bz)=0$ for all $\mu_h \in V_h$. Furthermore, in the first equation in \eqref{problem3discr} ``for all $\bv_h \in \bV_h$'' can be replaced by ``for all $\bv_h \in \bV_{h,\ast}$'',  since for all $\bv_h \in \Kzh$ we have $a(\bw_h,\bv_h)=b(\mu_h,\bv_h)=0$ for all $\bw_h \in \bV_h$, $\mu_h \in V_h$. Using this one easily checks that $(\bu_h, \lambda_h)$ is a solution of  \eqref{problem3discr}. \\
We investigate uniqueness. Let $(\bu_{h,1},\lambda_{h,1})$ and $(\bu_{h,2},\lambda_{h,2})$ be solutions of \eqref{problem3discr}.
For the differences $\delta \bu_h:=\bu_{h,1}-\bu_{h,2}$, $\delta \lambda_h=\lambda_{h,1}-\lambda_{h,2}$ we have that $(\delta \bu_h, \delta\lambda_h) \in \bV_{h,\ast} \times V_h$ satisfies
\[
 \begin{split}
  a(\delta \bu_h,\bv_h)+ b(\delta\lambda_h, \bv_h) &= 0 \quad \text{for all}~~\bv_h \in \bV_{h,\ast}\\
   b(\mu_h,\delta \bu_h) &= 0 \quad \text{for all}~~\mu_h \in  V_h.
\end{split}
\]
From the discrete ellipticity and  inf-sup properties it follows that $\delta \bu_h=0$ and $\delta \lambda_h=0$. 
\end{proof}

\begin{theorem}\label{Thmmain}
Let $h$ be sufficiently small such that $\Kzh=\Kz$ holds. 
Let $(\bu,\lambda)$ and $(\bu_h,\lambda_h)$ be the unique solutions of \eqref{problem3} and \eqref{problem3discr}, respectively. The following error bound holds: 
\begin{equation} \label{Mainest}
 \|\bu - \bu_h\|_1  \leq c_1 \min_{\bv_h \in \bV_{h}} \|\bu - \bv_h\|_1 + c_2\min_{\mu_h \in V_h}\|\lambda - \mu_h\|_1,
 \end{equation}
 with $c_1= (1 +\frac{1}{\gamma}\|a\|) (1+\frac{1}{\beta}\|b\|)$, $c_2= \frac{1}{\gamma} \|b\|$.
\end{theorem}
\begin{proof}
 Define $\be_h:=\bu-\bu_h$. Note that $Q_{\Kzh}\be_h=0$ and 
\begin{align}  \label{loc3}
  a(\be_h,\bv_h)+ b(\lambda - \lambda_h,\bv_h) & =0 \quad \text{for all}~ \bv_h \in \bV_h,\\
  b(\mu_h,\be_h) &=0 \quad \text{for all}~ \mu_h\in V_h. \label{loc31}
\end{align}
We take an arbitrary $\bv_h \in \bV_h$. From \eqref{discrinfsup} it follows, cf. \cite[Chapter II]{GR}, that there exists $\br_h \in \bV_{h,\ast}$ such that
\begin{equation} \label{GRresult}
 b(\mu_h,\br_h)=b(\mu_h,\bu-\bv_h) \quad \text{for all}~\mu_h \in V_h,~\text{and}~\|\br_h\|_1 \leq \frac{1}{\beta}\|b\|\|\bu-\bv_h\|_1.
\end{equation}
Define $\bw_h:=\br_h +\bv_h -Q_{\Kzh}(\br_h +\bv_h-\bz)$. Note that $Q_{\Kzh} \bw_h=Q_{\Kzh}\bz$ holds. Using $b(\mu_h,Q_{\Kzh}\bv)=0$ for all $\bv \in \bH^1(\Omega)$ and all $\mu_h \in V_h$ and with \eqref{loc31}-\eqref{GRresult} we obtain
\[
 b(\mu_h,\bw_h)=b(\mu_h,\br_h+\bv_h)=b(\mu_h,\bu)=b(\mu_h, \bu_h) \quad \text{for all}~\mu_h \in V_h.
\]
Hence, $\bw_h-\bu_h \in \ker_h(b)$ holds. Furthermore, $Q_{\Kzh}(\bw_h-\bu_h)=0$, and thus $\bw_h-\bu_h \in \bV_{h,\ast}$ holds. Thus we get $\bw_h-\bu_h \in \ker_h(b) \cap \bV_{h,\ast}$. Using this, the ellipticity estimate \eqref{ellip} and \eqref{loc3} we obtain, for arbitrary $\tau_h \in V_h$: 
\begin{align*}
 \|\bw_h-\bu_h\|_1^2 & \leq \frac{1}{\gamma} \big(a(\bw_h-\bu,\bw_h-\bu_h)+a(\bu-\bu_h,\bw_h-\bu_h)\big) \\
  & = \frac{1}{\gamma} \big(a(\bw_h-\bu,\bw_h-\bu_h) + b(\lambda_h - \lambda, \bw_h-\bu_h)\big) \\
  & = \frac{1}{\gamma} \big(a(\bw_h-\bu,\bw_h-\bu_h) + b(\tau_h - \lambda, \bw_h-\bu_h)\big) \\
  & \leq \frac{1}{\gamma} \big( \|a\| \|\bw_h-\bu\|_1 \|\bw_h-\bu_h\|_1 + \|b\| \|\tau_h - \lambda\|_1\| \bw_h-\bu_h\|_1\big).
\end{align*}
Using $\|\bu-\bu_h\|_1 \leq \|\bw_h-\bu\|_1+\|\bw_h - \bu_h\|_1$ we obtain
\begin{equation} \label{ll8}
 \|\bu-\bu_h\|_1 \leq (1 +\frac{1}{\gamma}\|a\|) \|\bw_h-\bu\|_1  +\frac{1}{\gamma} \|b\| \|\tau_h - \lambda\|_1.
\end{equation}
From $Q_{\Kzh}\bu=Q_{\Kz}\bu=Q_{\Kz}\bz=Q_{\Kzh}\bz=Q_{\Kzh}\bw_h$ we get $Q_{\Kzh}(\bw_h-\bu)=0$ and thus, using \eqref{GRresult}, we obtain
\begin{align*}
 \|\bw_h-\bu\|_1 & =\|(I-Q_{\Kzh})(\bw_h-\bu)\|_1=\|(I-Q_{\Kzh})(\br_h+\bv_h - \bu)\|_1  \\ 
  & \leq \|\bv_h-\bu\|_1+\|\br_h\|_1 \leq (1+\frac{1}{\beta}\|b\|)\|\bv_h - \bu\|_1. 
\end{align*}
Combining this with the result \eqref{ll8} yields
\[
 \|\bu-\bu_h\|_1 \leq (1 +\frac{1}{\gamma}\|a\|) (1+\frac{1}{\beta}\|b\|)\|\bv_h-\bu\|_1  +\frac{1}{\gamma} \|b\| \|\tau_h - \lambda\|_1.
\]
Since $\bv_h \in \bV_h$, $\tau_h \in V_h$ are arbitray, this completes the proof. 
\end{proof}
\ \\
We thus derived an optimal order bound for the discretization error $\|\bu-\bu_h\|_1$. 
A similar optimal bound as in \eqref{Mainest}, with different constants $c_1,c_1$, can be derived for the error $\|\lambda-\lambda_h\|_1$. Since in our application, this discretization error is less relevant we do not present this bound. \\
An $L^2$-norm estimate $\|\bu- \bu_h\|_{L^2} \leq c h\big(\|\bu-\bu_h\|_1+\|\lambda-\lambda_h\|_1\big)$ can be derived using the Aubin-Nitsche Lemma, cf. \cite[Secions 2.3.4, 2.4.2]{Ern04}. For this one needs $H^2$-regularity of the saddle point problem \eqref{problem3}. We are not aware of any regularity results for this type of saddle point problem.  

\begin{remark} \label{remrigid2} \rm
Assume we have an inverse level set problem as in Section~\ref{sectLS} with a  level set transport resulting from a rigid motion $\bu^\ast$ (cf. Remark~\ref{ex1}) and $\Kz=\{\mathbf{0}\}$. Then the exact solution of the saddle point problem  \eqref{problem3} is given by $(\bu,\lambda)=(\bu^\ast,0)$, cf. Remark~\ref{remrigid}. Furthermore, the rigid motion is a element of the finite element space $\bu^\ast \in \bV_h$. Hence, also after discretization we exactly recover the rigid motion. 
 \end{remark}

\section{Numerical experiments} \label{Experiments}

\newcommand{\errorpath}{./data}
We present results of numerical experiments.
First we  consider  the synthetic example presented in Lemma \ref{lemspecial}. Further experiments   show results for two inverse level set problems, namely an ellipse deforming to a biconcave shape and a rigid motion of an ellipse. Finally, motivated by the application treated in \cite{Tao2016}, we show an example in which the solution of an inverse level set problem is used to determine particle trajectories.  

%an ellipse deforming into a biconcave shape within a rectangular domain; the same rectangular domain with a central rectangle removed; and the domain adjusted by the removal of a central disc. Finally we look at an ellipse moving rigidly on a rectangular domain with a central disc removed.

We implemented the discretization \eqref{problem3discr} in Netgen/NGSolve\footnote{\url{https://ngsolve.org/}}\cite{ngsolve2024}. In our experiments, we start with an unstructured simplicial mesh  and refine the mesh globally using a marked-edge bisection method. We use continuous piecewise linear, quadratic or cubic  ($1 \leq k \leq 3$) finite elements. 

% For our experiments, we define the domain $\Omega = [-\frac43,\frac43]^2$. Additionally, we define the domain without a disc centered at the origin as $\tilde{\Omega} := \Omega \setminus \{x \in \R^2 \vert \Vert x \Vert_2 < 0.2\}$. This is necessary for most examples for Assumption \ref{assphi} to hold. Furthermore, we will always set $t=0$. In all experiments except for the one related to the regularity of the Lagrange multiplier, our approximated solution will be in $(\bV_h,V_h)=(\bP^k,P^k)$ for $k \in \{1,2\}$. We discretize the domain with an unstructured triangle mesh with varying mesh sizes in Netgen and solve the arising linear systems using the finite element library NGSolve\cite{schoberl2014c++}.

% In general, we will solve the discrete problem \eqref{problem2discr}, then compute the error relative to a reference solution with the $H^1$ norm on the corresponding domain. The reference solution is either known analytically or will be approximated by solving the saddle point problem on a fine mesh with polynomial degree $3$.

\FloatBarrier
\subsection{Artificial example of Lemma~\ref{lemspecial}}
We consider the problem stated in \\Lemma \ref{lemspecial}. We take
  $\Omega = (-\frac43,\frac43)^2$, $\hat{\bz} = \frac{1}{\sqrt{5}}\begin{pmatrix}
    1\\
    2
  \end{pmatrix}$
  , 
  $\hat{\bz}_\perp =\frac{1}{\sqrt{5}}\begin{pmatrix}
    -2\\
    1
  \end{pmatrix}$ and 
  \begin{align*}
      F(x)=\cos(x) - \frac{1}{\vert \Omega \vert}\int_\Omega \cos\left(\hat \bz \cdot \bx \right) d\bx.
  \end{align*}
We start with an initial mesh with (maximal) mesh size  $h=1$ and the mesh size is halved in every refinement.

As discussed in Lemma \ref{lemspecial}, there holds $\cK_{\hat \bz}={\rm span}(\hat{\bz}_\perp)$ and  thus  the matrix  corresponding to \eqref{problem3discr} is singular. We determine a solution $\tilde{\bu}_h$ of this discrete system and compute the  projection
$
  \bu_h = \tilde{\bu}_h-Q_{\cK_{\hat{\bz}}}\tilde{\bu}_h
$.

We use the analytical solution for the velocity given in Lemma \ref{lemspecial} as reference solution $\bu$ and show the error convergence rates in Fig. \ref{Synthetic}. Theory predicts convergence of order $\mathcal{O}(h^k)$ for $\Vert \bu-\bu_h\Vert_{1}$.  These optimal convergence rates are confirmed in the experiment. As an analytical solution for the Lagrange multiplier is not known, we do not show convergence rates for it.
\begin{figure}[htp]
  \centering
    \begin{tikzpicture}
      \begin{axis}[
          width=0.4\linewidth, height=0.4\linewidth,
          xlabel={Mesh size $h$},
              ylabel={Error},
              domain={0.04:1.0},
              legend style={
                  at={(0.5, -0.3)}, % Center the legend below the plot
                  anchor=north,
                  draw=none,
                  fill=none,
                  legend columns=2 % Distribute legend items in columns
              },
              ymode=log,
              xmode=log,
              % cycle list name=mark list
      ]
          \addplot[blue, mark=o] table[x=hmax, y=p2p2h1v] {\errorpath/Error4.dat};
          \addplot[black, mark=o] table[x=hmax, y=p1p1h1v] {\errorpath/Error4.dat};

          \addplot[dashed,line width=0.75pt] {1*x^1};
          \addplot[dotted,line width=0.95pt] {0.1*x^2};
          % \addplot[red,dashed,line width=0.95pt] {0.01*x^3};
          \legend{%
              \small $k=2$,
              \small $k=1$,
              \small $\cO(h^1)$,
              \small $\cO(h^2)$,
              % \small $h^3$,
          }
      \end{axis}
  \end{tikzpicture}
    \caption{Errors $\Vert \bu-\bu_h\Vert_{1}$ for the synthetic example}\label{Synthetic}
\end{figure}
\FloatBarrier

\subsection{Deforming ellipse}
In the following examples, we use the level set function
\begin{align*}
	\phi(x,y, t) &=t ((d^2 + x^2 + y^2 )^3 - 8d^2x^2-c^4)+(1-t)(1.3x^2+y^2-1)\overset{!}{=}0. \\
\end{align*}
with $	c = \frac{24}{25}$ and $d = \frac{19}{20}$, cf. Fig. \ref{biconcaveFig}.
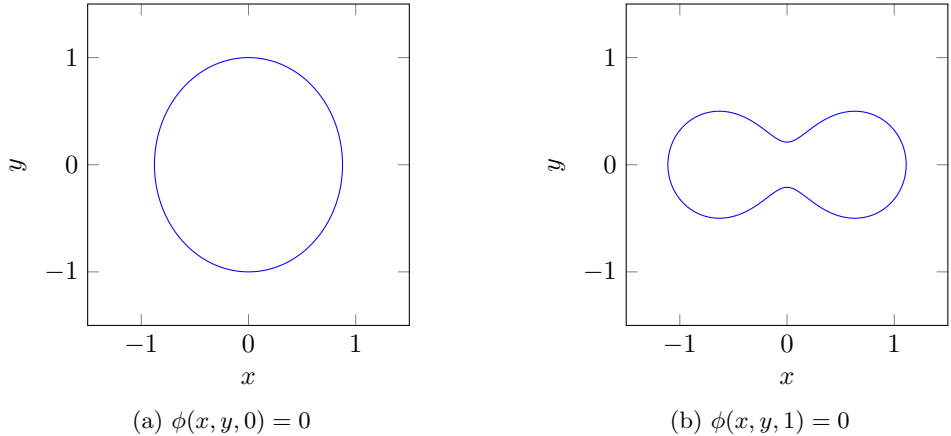
\begin{figure}[htp]
  \centering
  % First Subfigure
  \begin{subfigure}[b]{0.45\textwidth}
      \begin{tikzpicture}
          \begin{axis}[
              axis equal image,
              width=\linewidth, height=\linewidth,
              xlabel={$x$},
              ylabel={$y$},
              xmin=-1.5, xmax=1.5,
              ymin=-1.5, ymax=1.5,
          ]
              \addplot[blue] table {\errorpath/phi_level_set.dat};% table[col sep=space] {
                  %data/phi_level_set.dat
              %};
          \end{axis}
      \end{tikzpicture}
      \caption{$\phi(x,y, 0)=0$}
  \end{subfigure}
  \hfill
  % Second Subfigure
  \begin{subfigure}[b]{0.45\textwidth}
      \begin{tikzpicture}
          \begin{axis}[
              axis equal image,
              width=\linewidth, height=\linewidth,
              xlabel={$x$},
              ylabel={$y$},
              xmin=-1.5, xmax=1.5,
              ymin=-1.5, ymax=1.5,
          ]
              \addplot[blue] table {\errorpath/func_level_set.dat};
          \end{axis}
      \end{tikzpicture}
      \caption{$\phi(x,y, 1)=0$}
  \end{subfigure}
  \caption{Plots of the level sets}
  \label{biconcaveFig}
\end{figure}

For $t=0$ this corresponds to an ellipse and for $t=1$ to a biconcave shape. The prescribed velocity field is given by $\bz= - \frac{\partial \phi}{\partial t}\frac{1}{\|\nabla \phi\|} \hat \bz$, with $\hat{\bz}= \bn=\frac{\nabla \phi}{\Vert \nabla \phi\Vert}$. The only type of shape with rotational symmetry in two dimensions is a circle, so there holds $\cK_{\hat \bz}=\{\mathbf{0}\}$ for all $t \in [0,1]$.

 For small $t \geq 0$ there is only one critical point of $\nabla \phi$ close to the origin, while for larger $t$ we have multiple points with $\nabla \phi = 0$. We  solve \eqref{problem3discr} for $t=0$, but similar results are obtained if we take small $t > 0$. 
 
 For $\|\nabla \phi\| \to 0$ we do not have a ``stable'' normal velocity field. To avoid this instability we introduced the assumption \eqref{assphi}: $\min_{(\bx, t) \in \Omega_T}\|\nabla \phi(x,t)\| \geq c_{\min} >0$. This assumption is essential in the analysis of our method. In applications one can choose a computational domain that does not contain critical points of the level set function. However, when choosing such a  domain, one has to be aware of the fact that the smoothness of the solution depends on the smoothness of the domain boundary, cf. Remark~\ref{remLa}. We illustrate this effect in this example. In the subsections below we  consider three choices for the computational  domain: 1. no critical points of $\phi$ and the boundary is ``regular'' (interior boundary smooth, exterior boundary convex); 2. no critical points of $\phi$ and the interior boundary has reentrant corners; 3. the domain contains a critical point of $\phi$.
 
In the first two cases a reference solution $(\bu, \lambda)$ on a sufficiently fine mesh ($h=0.007$) with $k_{\text{ref}}=3$ is determined and we perform a convergence study starting with a coarse mesh ($h=1$) and refine it 5 times. 
 %For the last case, we compare the solution on the refined mesh from the second case to a solution on $(-\frac43,\frac43)$ with a mesh of similiar size.

\FloatBarrier
\subsubsection{Regular domain without critical points}\label{sectnocorner}
We define our domain as $\Omega = (-\frac43,\frac43)^2\setminus \{\bx \in \R^2 \vert \ \Vert \bx \Vert_2 \leq 0.2\}$. Removing the inner disc removes critical points of $ \phi$. Results are presented in   Fig. \ref{circleErrors}. The discrete velocity solution is shown in Fig.~\ref{GoodPlot}. 
\begin{figure}[ht!]
  \centering
  \begin{tikzpicture}
      \begin{groupplot}[
          group style={
              group size=2 by 1,
              horizontal sep=80pt
          },
          width=0.4\linewidth,
          height=0.4\linewidth,
          xlabel={Mesh size $h$},
          ylabel={Error},
          ymode=log,
          xmode=log,
          legend style={
              at={(-0.5,-0.9)}, % Place the legend above the plots
              anchor=south,
              draw=none,
              fill=none,
              legend columns=2, % Spread items across the width
              /tikz/every even column/.append style={column sep=5mm}
          },
          domain={0.02:1.0},
          cycle list name=mark list
      ]
  
      \nextgroupplot[title={$\Vert \bu-\bu_h\Vert_{1}$}]
      \addplot[blue, mark=o] table[x=hmax, y=p2p2h1v] {\errorpath/Error2.dat};
      \addplot[black, mark=x] table[x=hmax, y=p1p1h1v] {\errorpath/Error2.dat};
      \addplot[dashed,line width=0.75pt] {15*x^1};
      \addplot[dotted,line width=0.95pt] {4*x^2};
  
      \nextgroupplot[title={$\Vert \lambda-\lambda_h\Vert_{1}$}]
      \addplot[blue, mark=o] table[x=hmax, y=p2p2h1l] {\errorpath/Error2.dat};
      \addplot[black, mark=x] table[x=hmax, y=p1p1h1l] {\errorpath/Error2.dat};
      \addplot[dashed,line width=0.75pt] {15*x^1};
      \addplot[dotted,line width=0.95pt] {4*x^2};
\legend{
          $k=2$,
          $k=1$,
          $h^1$,
          $h^{2}$
      }
      \end{groupplot}
  \end{tikzpicture}
  \caption{$H^1$ errors for a regular domain without critical points of $\phi$}\label{circleErrors}
  \end{figure}
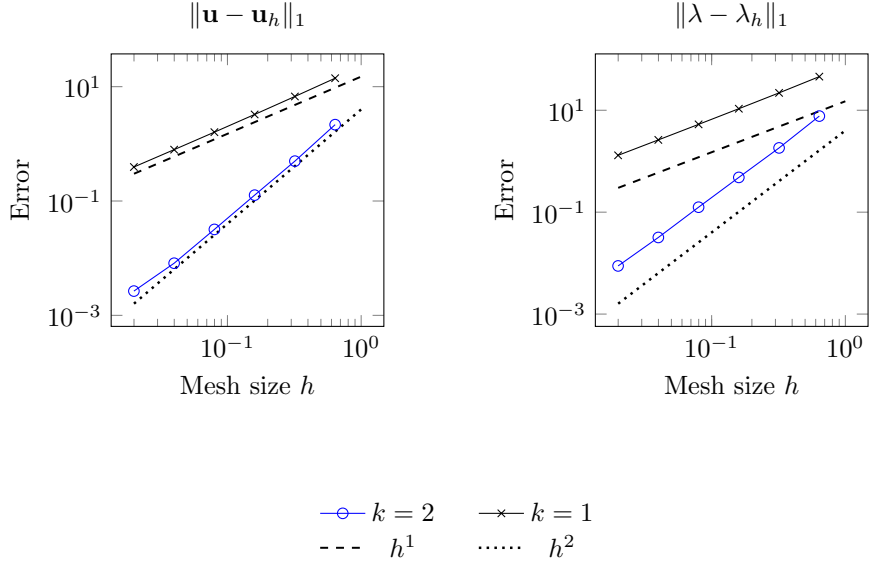
  
The expected order of convergence $\mathcal{O}(h^k)$ for $\Vert \bu-\bu_h\Vert_{H^1}$ and $\Vert \lambda-\lambda_h\Vert_{H^1}$ is clearly observed in the numerical results. 
  \FloatBarrier
\subsubsection{Domain with reentrant corners without critical points} \label{sectcorner}
We choose  the domain $\Omega= (-\frac43,\frac43)^2 \setminus [-0.4,0.4]^2$. We then have reentrant corners inside the domain, which may lead to poor regularity of $\lambda$ (and $u$), cf. Remark~\ref{remLa}. Results are presented in  Fig.~\ref{RegularityPlot}.

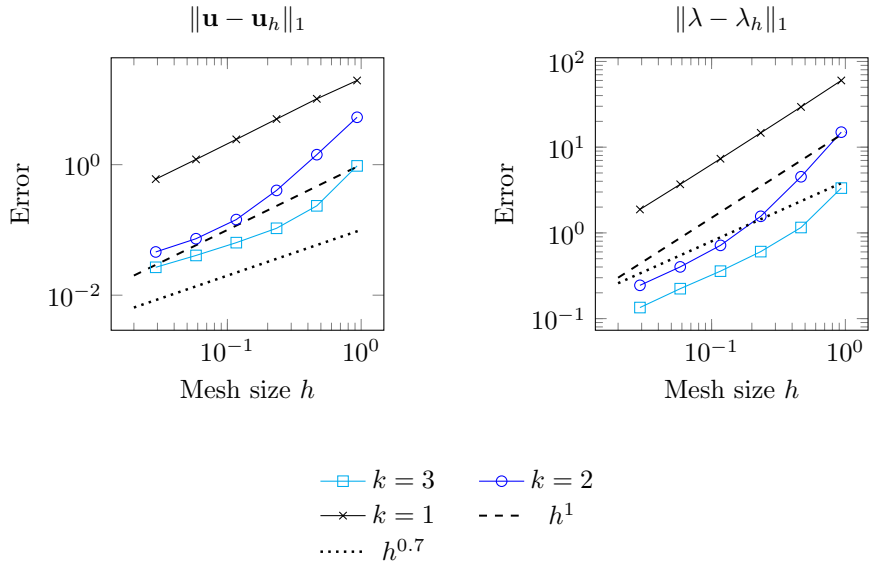
\begin{figure}[ht!]
  \centering
  \begin{tikzpicture}
      \begin{groupplot}[
          group style={
              group size=2 by 1,
              horizontal sep=80pt
          },
          width=0.4\linewidth,
          height=0.4\linewidth,
          xlabel={Mesh size $h$},
          ylabel={Error},
          ymode=log,
          xmode=log,
          minor x tick num=9, % Adds minor ticks between powers of 10
          minor y tick num=9,
          legend style={
              at={(-0.5,-0.9)}, % Place the legend above the plots
              anchor=south,
              draw=none,
              fill=none,
              legend columns=2, % Spread items across the width
              /tikz/every even column/.append style={column sep=5mm}
          },
          domain={0.02:1.0},
          cycle list name=mark list
      ]
  
      \nextgroupplot[title={$\Vert \bu-\bu_h\Vert_{1}$}]
      \addplot[cyan, mark=square] table[x=hmax, y=p3p3h1v] {\errorpath/Error6.dat};
      \addplot[blue, mark=o] table[x=hmax, y=p2p2h1v] {\errorpath/Error6.dat};
      \addplot[black, mark=x] table[x=hmax, y=p1p1h1v] {\errorpath/Error6.dat};
      \addplot[dashed,line width=0.75pt] {1*x^1};
      \addplot[dotted,line width=0.95pt] {0.1*x^0.7};
  
      \nextgroupplot[title={$\Vert \lambda-\lambda_h\Vert_{1}$}]
      \addplot[cyan, mark=square] table[x=hmax, y=p3p3h1l] {\errorpath/Error6.dat};
      \addplot[blue, mark=o] table[x=hmax, y=p2p2h1l] {\errorpath/Error6.dat};
      \addplot[black, mark=x] table[x=hmax, y=p1p1h1l] {\errorpath/Error6.dat};
      \addplot[dashed,line width=0.75pt] {15*x^1};
      \addplot[dotted,line width=0.95pt] {4*x^0.7};
  
      \legend{
          $k=3$,
          $k=2$,
          $k=1$,
          $h^1$,
          $h^{0.7}$
      }
      \end{groupplot}
  \end{tikzpicture}
  \caption{$H^1$ errors for a domain with reentrant corners without critical points of $ \phi$}\label{RegularityPlot}
  \end{figure}
  
  The results indicate that optimal convergence of order $\mathcal{O}(h^k)$ for $\Vert \bu -\bu_h\Vert_1$ and $\Vert \lambda-\lambda_h \Vert_{H^1}$ is \emph{not} reached. This is probably due to a loss of regularity. This claim is supported by the plot of the pointwise errors shown  in Fig.~\ref{cornerErrorPlot}.
\begin{figure}[ht!]
  \centering

  \begin{subfigure}[b]{0.45\textwidth}
    \centering
    \includegraphics[width=\textwidth]{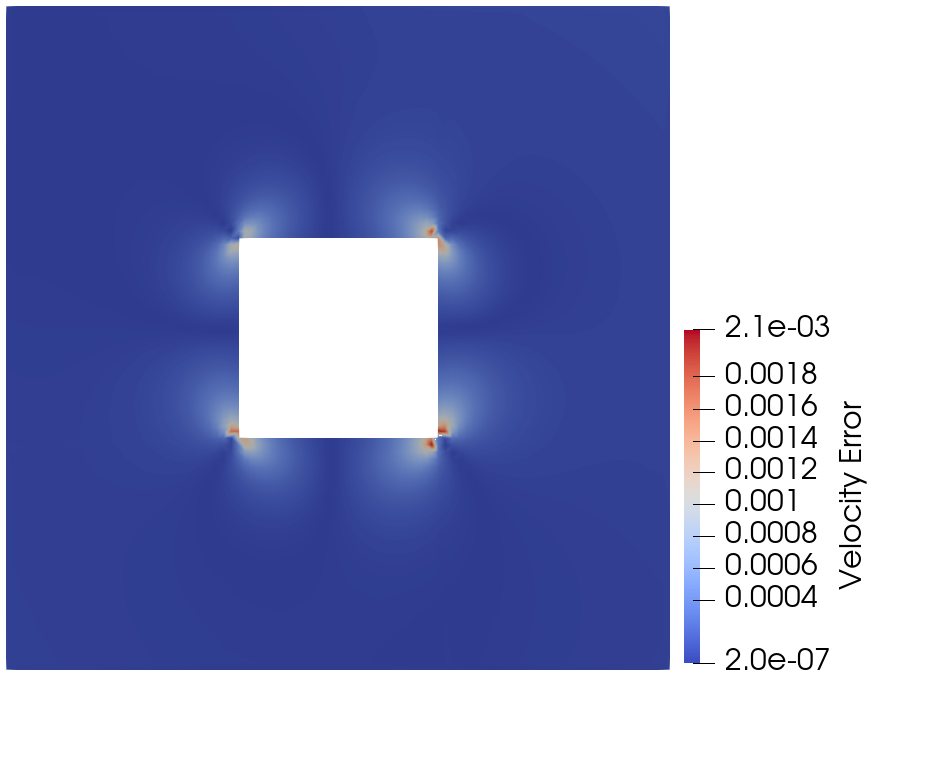}
    \caption{$\Vert \bu(\bx)-\bu_h(\bx)\Vert$}\label{VelocityErrorPlot}
\end{subfigure}
  \hfill
  \begin{subfigure}[b]{0.45\textwidth}
    \centering
      \includegraphics[width=\textwidth]{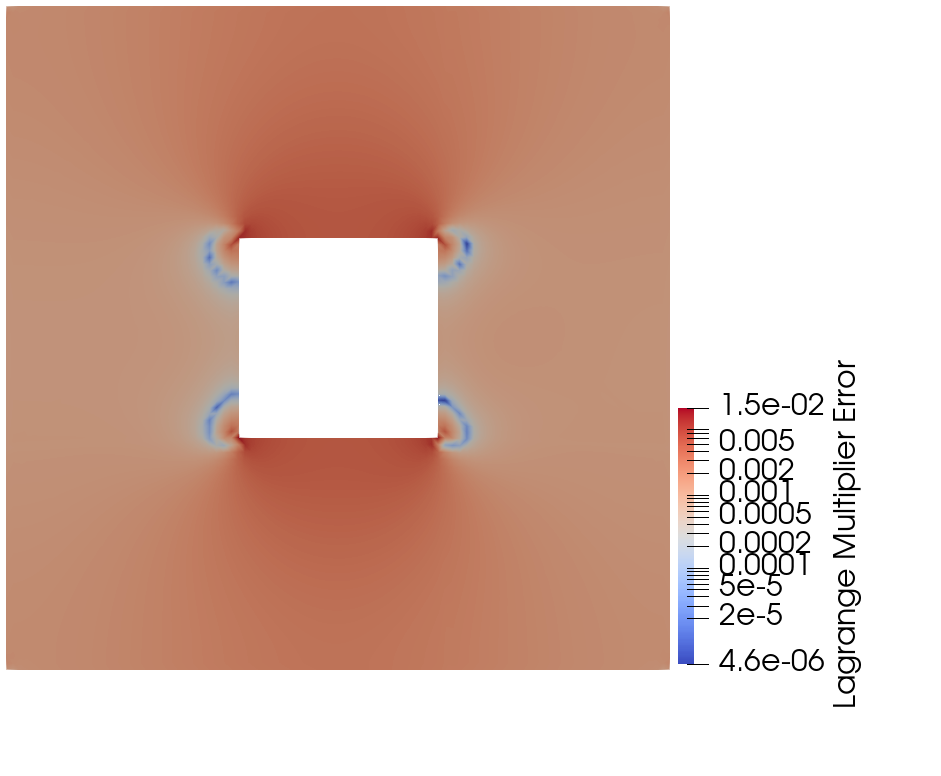}
      \caption{$\vert \lambda(\bx)-\lambda_h(\bx)\vert$}\label{LagrangeErrorPlot}
  \end{subfigure}
\caption{Pointwise errors in the velocity and Lagrange multiplier}
\label{cornerErrorPlot}
\end{figure}

\FloatBarrier
\subsubsection{Domain with a critical point of $\phi$}
We take $\Omega =  (-\frac43,\frac43)^2$, a domain that contains a critical point: $\nabla \phi(0,0,0) = 0$. 
In Fig.~\ref{GoodPlot} we show the discrete velocity solution for the case of a regular domain without critical points, cf. subsection~\ref{sectnocorner} above.
The discrete  solution for the domain that contains a critical point is shown in  Fig.~\ref{BadPlot}. One clearly observes a strange velocity field with (nearly) singular points at the center and the south pole of the ellipse. This result shows that critical points of the level set function cause difficulties and  that assumption \eqref{assphi} is essential. 
\begin{figure}[ht!]
  \centering

  \begin{subfigure}[t]{0.49\textwidth}
    \centering
    \includegraphics[width=\textwidth]{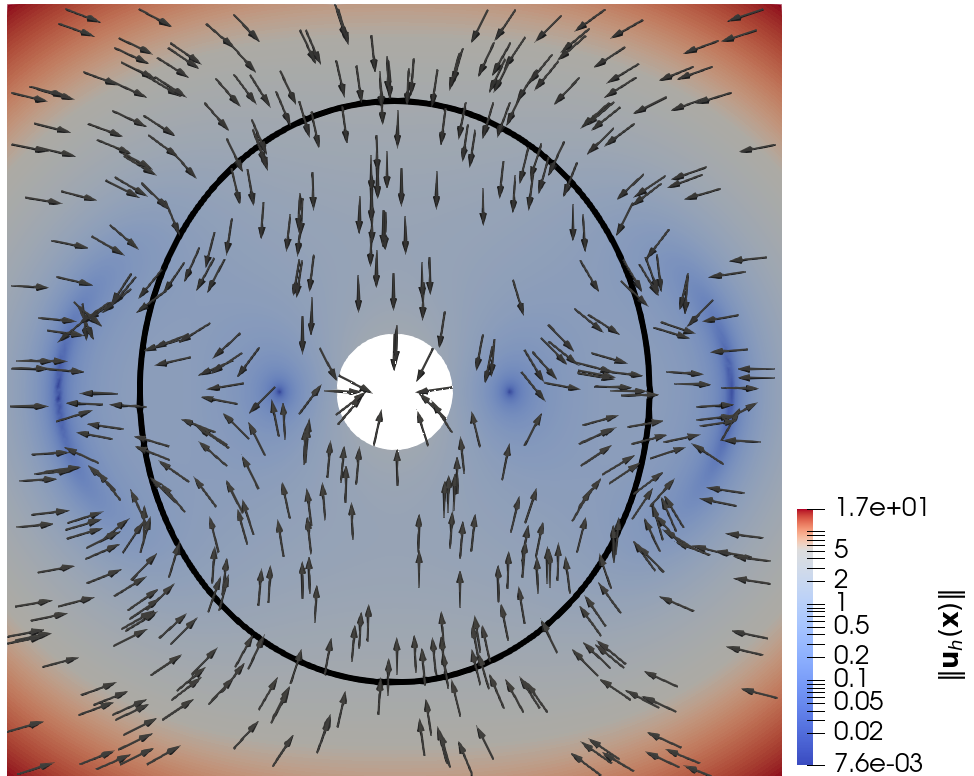}
    \caption{Regular domain without a critical point of $\phi$}\label{GoodPlot}
\end{subfigure}
  \hfill
  \begin{subfigure}[t]{0.49\textwidth}
    \centering
      \includegraphics[width=\textwidth]{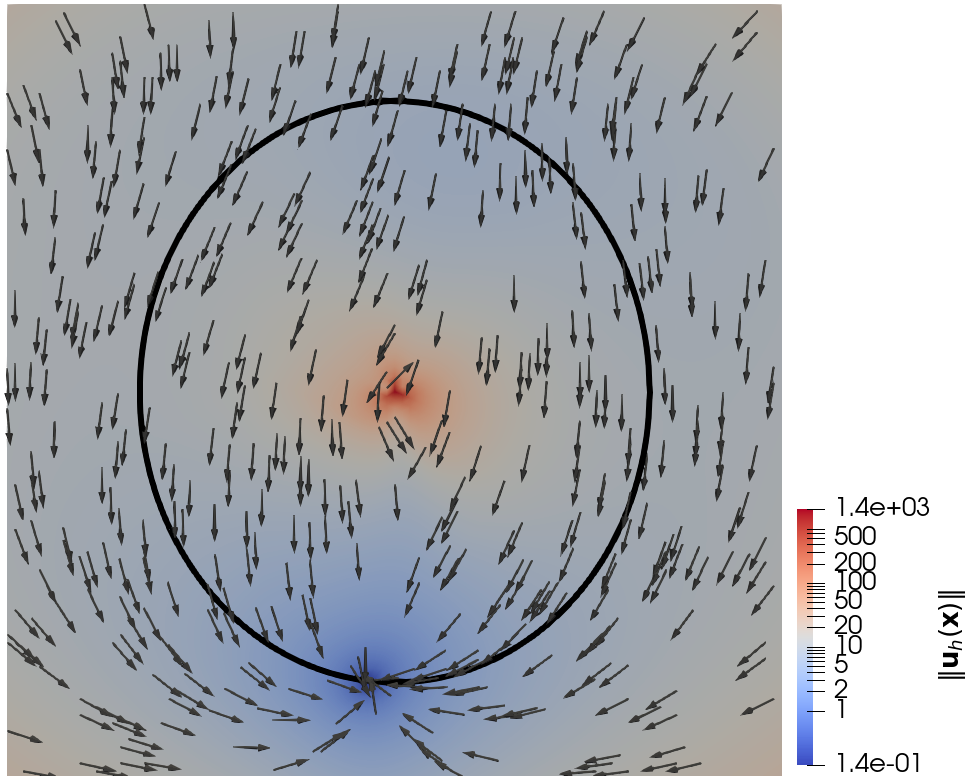}
      \caption{Domain with critical point of $\phi$}\label{BadPlot}
  \end{subfigure}
\caption{$\bu_h(\bx)$ on a domain without (left)  and with (right) a critical point of $\phi$}
\label{criticalComparison}
\end{figure}
  \FloatBarrier
\subsection{A rigid motion example}
We consider the example of an ellipse undergoing a rigid motion (rotation), cf. Remark~\ref{remrigid2}. The  domain is  $\Omega = (-\frac43,\frac43)^2\setminus \{\bx \in \R^2 \vert \ \Vert \bx \Vert_2 \leq 0.2\}$ and we use

\begin{align*}
  \phi(x,y,t) &= 1.3 \hat{x}(x,y,t)^2+\hat{y}(x,y,t)^2-1
\end{align*}
with
\begin{align*}
  \begin{pmatrix}\hat{x}(x,y,t)\\
  \hat{y}(x,y,t)
  \end{pmatrix} &= M(t) \begin{pmatrix}
    x\\
    y
  \end{pmatrix}, \quad
  M(t) = \begin{pmatrix}
		\cos(0.1t) & -\sin(0.1t) \\
		\sin(0.1t) & \cos(0.1t)
	\end{pmatrix}.
\end{align*}
The  domain does not contain  a critical point of $\phi$.  Furthermore, $\cK_{\hat{\bz}} = \{\mathbf{0}\}$ for all $t \in [0,1]$ as the level set function describes a non-circular ellipse. This level set transport can be obtained from a rigid motion (rotation) and the solution of the strain energy minimization problem~\ref{problem1} is given by, cf. \cite{Tao2016},
\begin{align*}
  \bu(\bx) &= ((\frac{\partial}{\partial t} M^T) M)\big|_{t=0} \begin{pmatrix}
    x\\
    y
  \end{pmatrix}.
\end{align*}
In the saddle point formulation the Lagrange multiplier solution is $\lambda=0$. Note that 
$\bu$ is linear and  lies in our finite element space. Hence, in the discretization we should recover  the exact solution  $\bu$ up to machine accuracy, cf.~Remark \ref{remrigid2}. This statement, however, assumes exact integration during assembly of the system matrix. As $\hat{\bz}$ is not in our finite element space, the discrete problem is perturbed due to numerical quadrature. NGSolve uses by default quadrature rules of order $2k$, but this can be increased with a parameter $l$ to $2k+l$. Results for $k=1$ and varying $l$ are shown in  Fig.~\ref{rigidMotion}.
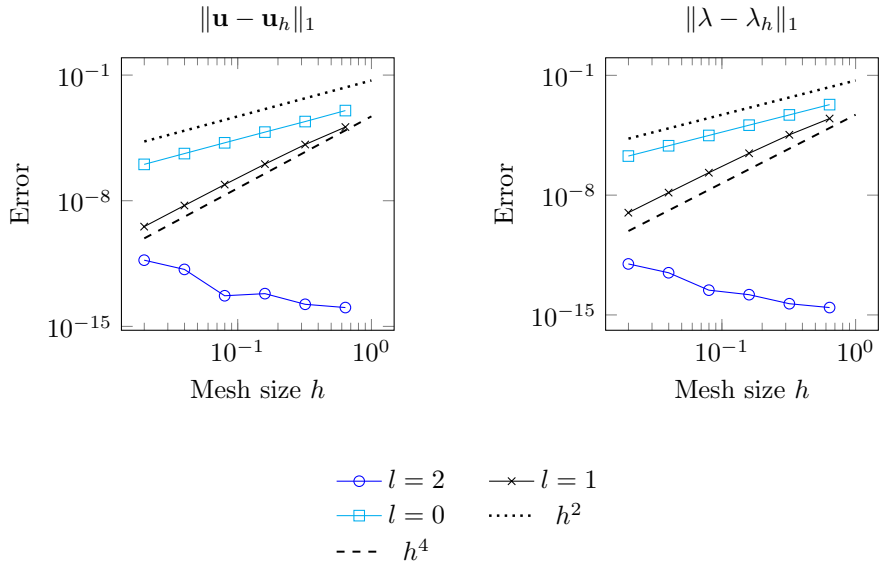
\begin{figure}[ht!]
  \centering
  \begin{tikzpicture}
      \begin{groupplot}[
          group style={
              group size=2 by 1,
              horizontal sep=80pt
          },
          width=0.4\linewidth,
          height=0.4\linewidth,
          xlabel={Mesh size $h$},
          ylabel={Error},
          ymode=log,
          xmode=log,
          legend style={
              at={(-0.5,-0.9)}, % Place the legend above the plots
              anchor=south,
              draw=none,
              fill=none,
              legend columns=2, % Spread items across the width
              /tikz/every even column/.append style={column sep=5mm}
          },
          domain={0.02:1.0},
          cycle list name=mark list
      ]
  
      \nextgroupplot[title={$\Vert \bu-\bu_h\Vert_{1}$}]
      \addplot[blue, mark=o] table[x=hmax, y=p2p2h1v] {\errorpath/Error1.dat};
      \addplot[black, mark=x] table[x=hmax, y=p1p1h1v] {\errorpath/Error1.dat};
      \addplot[cyan, mark=square] table[x=hmax, y=p0p0h1v] {\errorpath/Error1.dat};
      
      % \addplot[dashed,line width=0.75pt] {0.2*x^1};
      \addplot[dotted,line width=0.95pt] {0.05*x^2};
      \addplot[dashed,line width=0.75pt] {0.0005*x^4};
      \nextgroupplot[title={$\Vert \lambda-\lambda_h\Vert_{1}$}]
      \addplot[blue, mark=o] table[x=hmax, y=p2p2h1l] {\errorpath/Error1.dat};
      \addplot[black, mark=x] table[x=hmax, y=p1p1h1l] {\errorpath/Error1.dat};
      \addplot[cyan, mark=square] table[x=hmax, y=p0p0h1l] {\errorpath/Error1.dat};

      \addplot[dotted,line width=0.95pt] {0.05*x^2};
      \addplot[dashed,line width=0.75pt] {0.0005*x^4};
  
      \legend{
          $l=2$,
          $l=1$,
          $l=0$,
          % $h^1$,
          $h^{2}$,
          $h^{4}$
      }
      \end{groupplot}
  \end{tikzpicture}
  \caption{$H^1$ errors for an ellipse undergoing rigid motion.}\label{rigidMotion}
  \end{figure}
  
  With increasing accuracy of the quadrature, the error both for the velocity field and the Lagrange multiplier decreases. For sufficiently accurate quadrature ($l=2$) machine  accuracy is reached, for not too small mesh size $h$. For very small mesh sizes $h$ there appears to be some stability issue, that may be related to the direct solver used for the linear saddle point system. 
\subsection{Computing trajectories}

We consider the problem of near-isometric level set tracking introduced in \cite{Tao2016}.
Let $\phi:\Omega_T \rightarrow \R$ be a level set function and $\bxi_0 \in \Gamma(0)$ a point on the zero level set of $\phi$ at $t=0$. Further, let $\bu_t$ be the solution to \eqref{problem1} for a fixed $t \in [0,t_e]$. The particle tracking problem reads as follows. Find $\bxi(t): [0,t_e] \rightarrow \Omega$, such that
\begin{equation} \label{ODE} \begin{split}
  \frac{d }{d t}\bxi &= \bu_{t}(\bxi), \\
  \bxi (0) &= \bxi_0.
\end{split}
\end{equation}
We can approximate the solution $\bu_t$ of \eqref{problem1}, for any given $t \in [0,t_e]$, by solving \eqref{problem3discr}. For simplicity, we use an explicit Euler time stepping scheme for time discretization in \eqref{ODE}. Given $N \in \mathbb{N}$, we define $\Delta t := \frac{t_e}{N}$, $t_i :=i \Delta t$ and $i \in \{0,\dots,N\}$. Then we construct, given a starting point $\bxi_{0} \in \Gamma(0)$, points $\{\bxi_{h,i}\vert \ i \in \{0,\dots,N\}\} \subset \R^d$ using 
\begin{align}\label{trackingdiscr}
    \bxi_{h,i} &= \bxi_{h,i-1}+\Delta t\ \bu_{h,t_{i-1}}(\bxi_{h,i-1})\quad \text{for}~ i=1,\ldots N,\\
 \nonumber \bxi_{h,0} &= \bxi_0,
\end{align}

with $\bu_{h,t_i}$ being the solution to \eqref{problem3discr} at time $t_i$.

For comparison, we construct $\{\tilde{\bxi}_{h,i}\vert \ 0 \leq i \leq N\} \subset \R^d$ as solutions of 
\begin{align}\label{baseline}
    \tilde{\bxi}_{h,i} &= \tilde{\bxi}_{h,i-1}+\Delta t\ \tilde{\bu}_{t_{i-1}}(\tilde{\bxi}_{h,i-1}) \quad \text{for}~ i\in\{1,\dots N\},\\
 \nonumber \tilde{\bxi}_{h,0} &= \bxi_0,
\end{align} 

with $\tilde{\bu}_{t_i}:=-\frac{\partial \phi}{\partial t}\frac{1}{\Vert \nabla \phi \Vert^2}\nabla \phi$, i.e. a purely normal velocity field without any contribution tangential to the level set.

\subsubsection{Particle transport on a surface in $\R^3$}
We consider a level set representing a capsule shape and apply a bending deformation to it.

Let $\bx = \begin{pmatrix}
  x\\y\\z
\end{pmatrix}$. Defining
\begin{align*}
  C(\bx) &= \left\| \bx - \begin{pmatrix}
    0\\\max(-0.7, \min(0.7, y))\\0
  \end{pmatrix}\right\| -0.3
\end{align*}
and
\begin{align*}
  B(\bx, t) &= \begin{pmatrix}
    \cos(ty) & -\sin(ty) & 0\\
    \sin(ty) & \cos(ty) & 0\\
    0 & 0 & 1.02
  \end{pmatrix}  \bx
\end{align*}

we obtain the level set function for a bending capsule via
\begin{align*}
  \phi(\bx, t) := C\big(B (\bx, t)\big).
\end{align*} 
Note that the scaling factor of $1.02$ in $B(\cdot, \cdot)$ ensures that there are no symmetries in the initial configuration.

Define $\varphi = \frac{2\pi}{60}$ and $\theta = \frac{\pi}{80}$. We sample $M=4800$ points $\bp_i$ on a sphere with radius $r=1.2$ using 
\begin{align*}
  \bp_{60k+j}=r\begin{pmatrix}
     \sin(k\theta)\cos(j\varphi)\\
    \cos(k\theta)\\
    \sin(k\theta)\sin(j\varphi)
  \end{pmatrix}, \quad 0\leq k <80,\ 0\leq j < 60,
\end{align*} 
and project these orthogonally onto the level set for $t=0$.  This results in  a point cloud $(\bxi_0^i)_{i=1}^M$ on $\Gamma(0)$. Rendering these points as spheres with radius $0.01$, with RGB coloring depending on the coordinate values in Blender\cite{Blender}, we get Figure \ref{particleinitial}.

Now, we solve \eqref{baseline} and \eqref{trackingdiscr} on the time interval $[0,1]$ with $N=60$. 
The normal velocity $\tilde{\bu}_{t_i}$, used in  \eqref{baseline}, can be determined analytically as $\phi$ is known. 
The resulting particle distribution at $t=1$ is shown in Fig.~\ref{particlenormal}.
While the particles do track the bending shape, the particle density on the front is significantly lower and particles accumulate on the back side of the bent shape. 

For solving \eqref{trackingdiscr}, we create a quasi-uniform tetrahedral mesh with mesh size $h \approx 0.0875$ on $\Omega := (-1.3,1.3)^3$. For efficiency reasons and to avoid critical points of $\nabla \phi$, we only solve the problem on a narrow band containing all tetrahedra where $-0.02 \leq \phi \leq 0.02$ using ngsxfem\cite{LHPvW21}.  The resulting particle distribution at $t=1$ is shown 
in Fig.~\ref{particlemethod}. The result shows that, as a result of the near-isometric velocity field construction,  the distances between particles are approximately constant during the shape evolution. 
\begin{figure}[h!]
  \centering
  \includegraphics[width=0.9\linewidth]{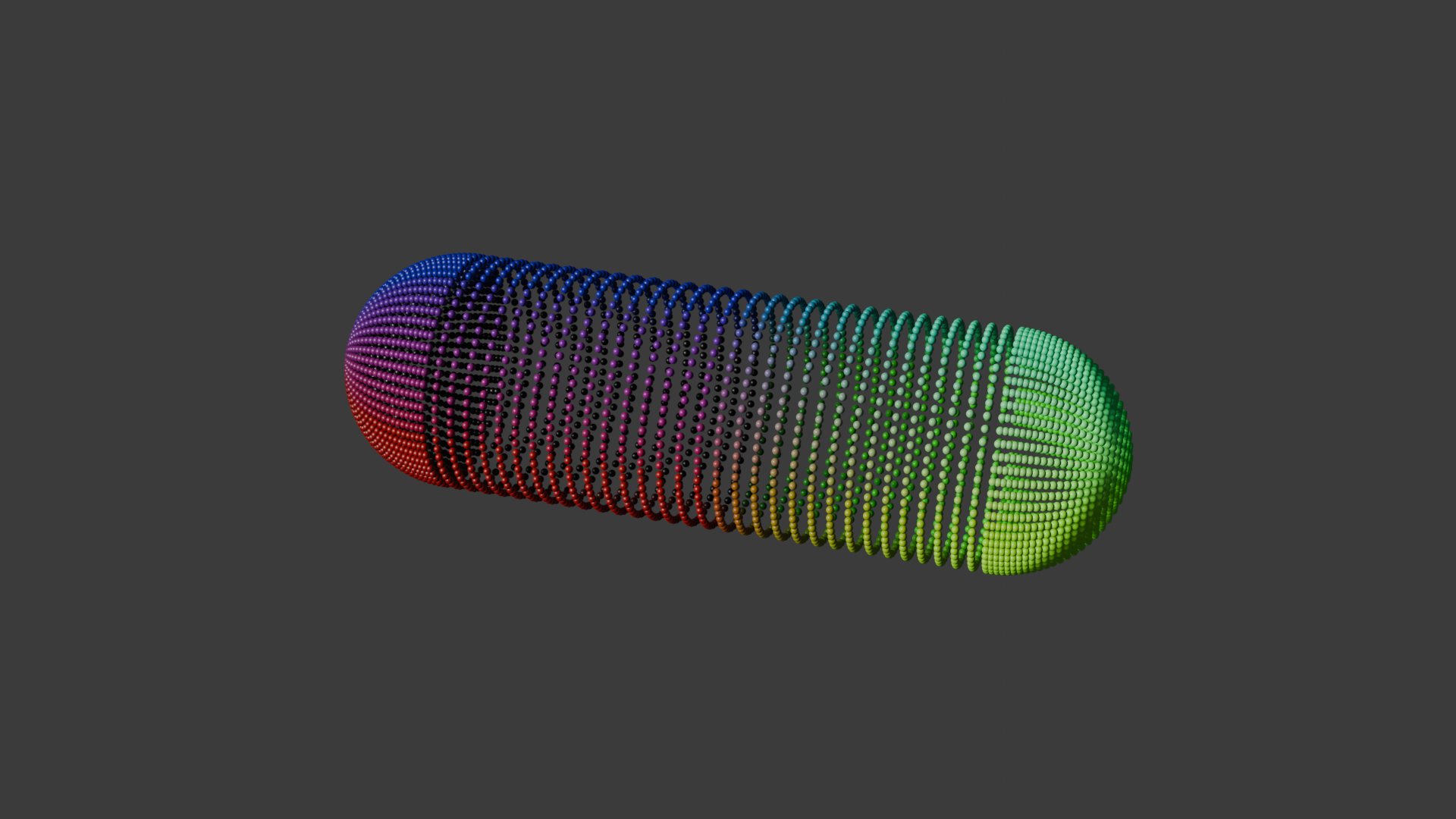}
  \caption{Initial particle configuration}\label{particleinitial}
\end{figure}

\begin{figure}[h!]
  \centering
  \begin{subfigure}[t]{0.45\linewidth}
    \includegraphics[width=\linewidth]{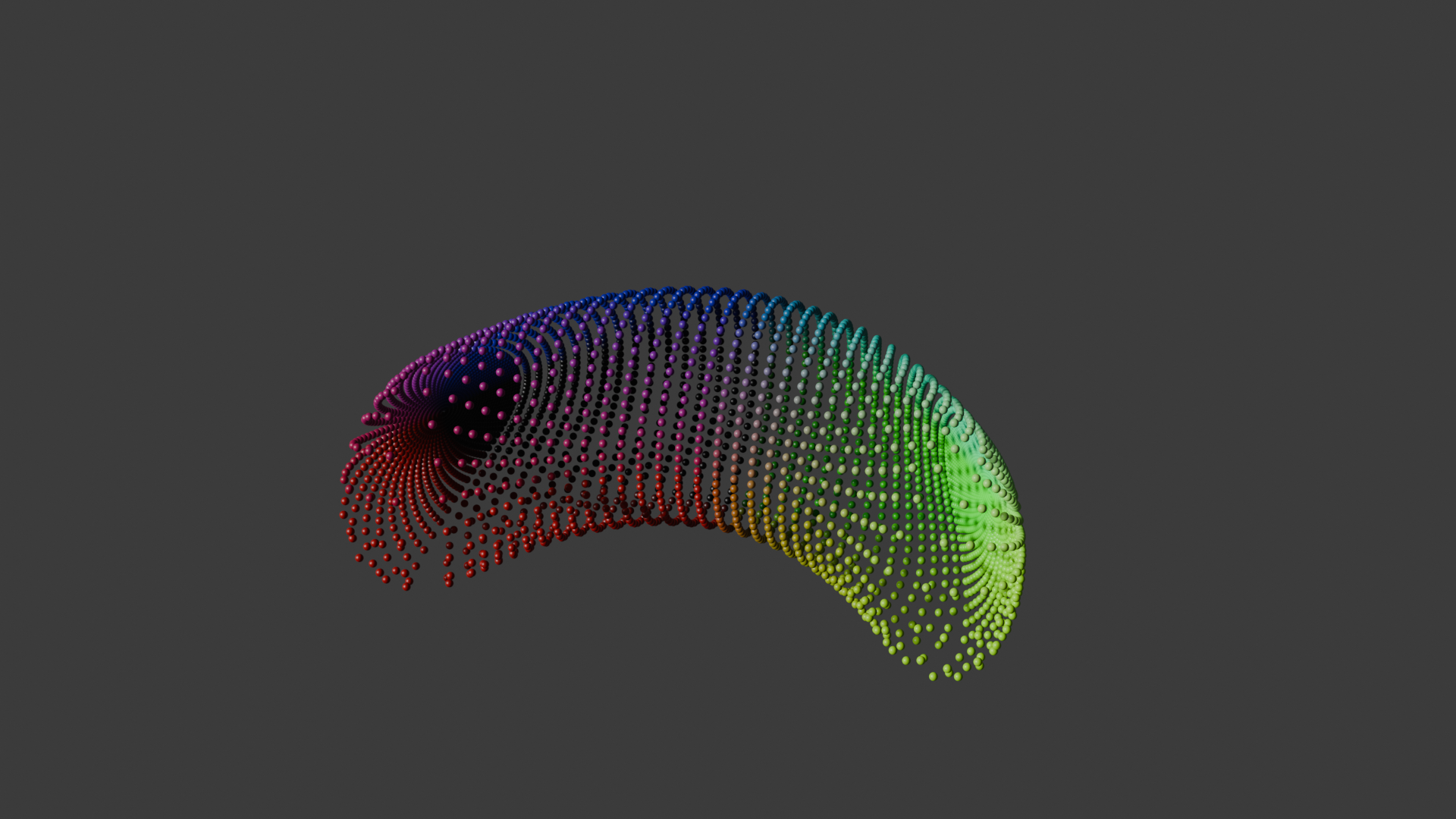}
    \caption{Particle tracking with normal velocity}\label{particlenormal}

  \end{subfigure}
    \begin{subfigure}[t]{0.45\linewidth}
      \centering

    \includegraphics[width=\linewidth]{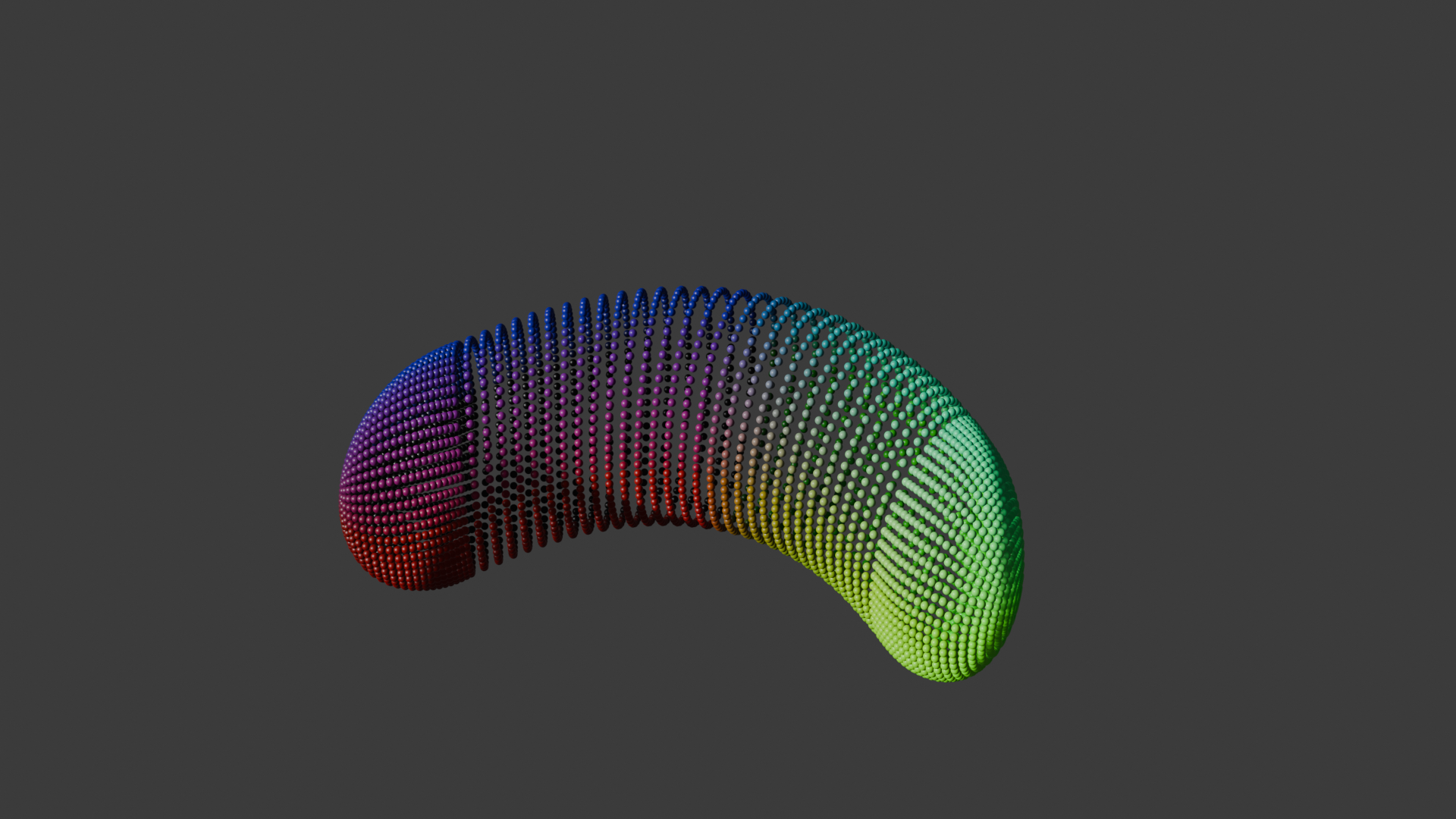}
    \caption{Near-isometric tracking}\label{particlemethod}

    \end{subfigure}
    \caption{Particle configuration at $t=1$}
\end{figure}
%   \begin{figure}[h!]
%     \centering

%     \includegraphics[width=0.9\linewidth]{\errorpath/Method.png}
%     \caption{Particle configuration with near-isometric tracking}\label{particlemethod}

% \end{figure}

\FloatBarrier
\bibliographystyle{siam}
\bibliography{literatur}{}

\appendix

%\section{Appendix}

%\section{Properties of $\cK_{\hat \bz}$}
\section{Proof of Lemma~\ref{lemkernel1}}
  Note that the colums of $R_\bx:=\begin{pmatrix} 1 & 0 & -y \\ 0 & 1 & x \end{pmatrix}$ form a basis of the space $K$ of rigid motions.    The kernel space $\cK_{\hat \bz}$ is a linear subspace of $K$. Hence, there is a linear subspace $ U \subset \R^3$ with $\dim (U) =\dim (\cK_{\hat \bz})$ and  such that $\bw \in \cK_{\hat \bz}$ iff $\bw(\bx)= R_\bx \bp$ for a $\bp \in U$ and all $\bx \in \Omega$. 
For all $\bv \in \cK_{\hat \bz}$ we have  $\bv \cdot \hat \bz=0$ a.e. on $\Omega$. This yields
\[
 (R_\bx \bp)^T \hat \bz(\bx) = \bp^T (R_\bx^T\hat \bz(\bx))=0 \quad \forall ~\bp \in U, ~\forall ~\bx \in \Omega ~\text{(a.e.)}.
\]
Hence we obtain $U \perp W_{\hat \bz}$ with 
$W_{\hat \bz}:= {\rm span} \{ \, \cup_{\bx \in \Omega} R_\bx^T \hat \bz(\bx)\}$. 
 Note that, with $\hat \bz =( \hat z_1, \hat z_2)^T$, we have
 \[
  R_\bx^T \hat \bz(\bx)= \begin{pmatrix} \hat z_1(\bx) \\ \hat z_2(\bx) \\ - y \hat z_1(\bx) + x \hat z_2(\bx) \end{pmatrix}.
 \]
 Using $\|\hat \bz \|=1$ on $\Omega$ it follows that $\dim (W_{\hat \bz}) \geq 2$, which implies $\dim (U) \leq 1$. Thus the first statement of the Lemma~\ref{lemkernel1} holds. From $\dim(W_{\hat \bz})=3$ it follows that $\dim(U)=0$, which proves the second statement of the lemma. 
Let $\bw \in  K$ be a nontrivial rigid motion in $\R^2$ and $\bw_\perp$ such that $\bw_\perp \cdot \bw=0$ on 
$\R^2$. If $\bw$ is a rotation we assume that the center of the rotation does not lie in $\overline{\Omega}$. For $\hat \bz:= \frac{ \bw_\perp}{\|\bw_\perp\|}$ we then have $\hat \bz \in C^\infty(\overline{\Omega})$. Per construction we have $\bw \in \cK_{\hat \bz}$ and since $\dim (\cK_{\hat \bz}) \leq 1$ it follows that $\cK_{\hat \bz}= {\rm span} (\bw)$. This completes the proof.

\section{Proof of Lemma~\ref{lemkernel2}}
The proof of the first part of the first two statements in the lemma is along the same lines as for the case $d=2$ above.
The columns of the matrix $R_\bx$ in \eqref{defR} form  a basis of the space $K$ of rigid motions.    The kernel space $\cK_{\hat \bz}$ is a linear subspace of $K$. Hence, there is a linear subspace $ U \subset \R^6$ with $\dim (U) =\dim (\cK_{\hat \bz})$ and  such that $\bw \in \cK_{\hat \bz}$ iff $\bw(\bx)= R_\bx \bp$ for a $\bp \in U$ and all $\bx \in \Omega$. 
For all $\bv \in \cK_{\hat \bz}$ we have  $\bv \cdot \hat \bz=0$ a.e. on $\Omega$. This yields
\[
 (R_\bx \bp)^T \hat \bz(\bx) = \bp^T (R_\bx^T\hat \bz(\bx))=0 \quad \forall ~\bp \in U, ~\forall ~\bx \in \Omega ~\text{(a.e.)}.
\]
Hence we obtain $U \perp W_{\hat \bz}$ with 
$W_{\hat \bz}:= {\rm span} \{ \, \cup_{\bx \in \Omega} R_\bx^T \hat \bz(\bx)\}$. 
 Note that, with $\hat \bz =( \hat z_1, \hat z_2,\hat z_3)^T$, we have
 \[
  R_\bx^T \hat \bz(\bx)= \begin{pmatrix} \hat z_1(\bx) \\ \hat z_2(\bx) \\ \hat z_3(\bx) \\  - y \hat z_1(\bx) + x \hat z_2(\bx) \\ -z \hat z_1(\bx) + x \hat z_3(\bx) \\  - z \hat z_2(\bx) + y \hat z_3(\bx) \ \end{pmatrix}.
 \]
 Using $\|\hat \bz \|=1$ on $\Omega$ it follows that $\dim (W_{\hat \bz}) \geq 3$, which implies $\dim (U) \leq 3$. Thus the first statement of the Lemma~\ref{lemkernel2} holds. From $\dim(W_{\hat \bz})=6$ it follows that $\dim(U)=0$, which proves the second statement of the lemma.  
The proof of \eqref{resdd} is elementary and left to the reader.
%Take $\bp \in \R^3 \setminus \{0\}$, $\bz=z_1 \bp$ with $z_1 \in H^1(\Omega) \setminus\{0\}$. Define $K_\bp^e=\{\, \bw^e= \bw \circ P_{\bp_\perp}~|~\bw \in K_\bp\,\} \subset K$. Take $\bw^e \in  K_\bp^e$. Since $\bw \cdot \bp=0$ we have $ \bw^e \cdot \bp = (\bw \circ P_{\bp_\perp})\cdot \bp =0$ on $\Omega$. Hence $K_\bp^e \subset V_{\hat \bz}^0$, and due to $K_\bp^e \subset K$ we have $K_\bp^e \subset \cK_{\hat \bz}$. Using $\dim (K_\bp^e)=3$ it follows that $K_\bp^e= \cK_{\hat \bz}$, which proves the result \eqref{resdd}.
%\ \\[1ex] 
\section{Proof of Lemma~\ref{lemspecial}}
For $\bz$ as in \eqref{defz} we have $\cK_{\hat \bz}= {\rm span}(\bgp)$, cf. Lemma~\ref{lemkernel1}.
The solution of \eqref{problem1} is the unique solution of the variational problem:
\begin{equation} \label{ppp} \begin{split}
  & \bv \in V_{\hat \bz}^0 ~~\text{with}~~(\bv,\bgp)_1= \int_\Omega \bv \cdot \bgp\, d\bx =0 ~~\text{such that} \\
  & a(\bv,\bw)= - a(\bz,\bw)\quad \text{for all}~\bw \in V_{\hat \bz}^0.
\end{split}
\end{equation}
For $\bv, \bw \in \bH^1(\Omega)$ we write $\bv=v_1 {\hat \bz} +v_2\bgp$, $\bw=w_1 {\hat \bz}+w_2 \bgp$.
Note that $\nabla \bv + \nabla\bv^T = \nabla v_1 {\hat \bz}^T + {\hat \bz} \nabla v_1^T +  \nabla v_2 \bgp^T + \bgp \nabla v_2^T$.  Using this and $\bp_1 \bq_1^T : \bp_2 \bq_2^T = (\bp_1 \cdot \bp_2)(\bq_1\cdot \bq_2)$ for $\bp_i, \bq_i \in \R^2$, we obtain
\begin{align*}
 E(\bv):E(\bw)& = 2 \nabla v_1\cdot \nabla w_1 + 2\nabla v_2\cdot \nabla w_2 \\
              & + 2(\nabla v_1 \cdot {\hat \bz})(\nabla w_1 \cdot {\hat \bz}) + 2(\nabla v_2 \cdot \bgp)(\nabla w_2 \cdot \bgp) \\
              & + 2(\nabla v_1 \cdot \bgp)(\nabla w_2 \cdot {\hat \bz}) + 2(\nabla v_2 \cdot {\hat \bz})(\nabla w_1 \cdot \bgp).
\end{align*}
Since $\bz$ is of the form $\bz=z_1 {\hat \bz}$ we have $V_{\hat \bz}^0=\{\, v_2 \bgp~|~v_2 \in H^1(\Omega)\,\}$. For $\bv=v_2 \bgp \in V_{\hat \bz}^0$, $\bw =w_2 \bgp \in V_{\hat \bz}^0$, we get, using $\nabla v_2= (\nabla v_2 \cdot {\hat \bz}) {\hat \bz}+(\nabla v_2 \cdot \bgp) \bgp$,
\[ \begin{split}
 a(\bv,\bw)&= 2 \int_{\Omega}\nabla v_2 \cdot \nabla w_2 +  (\nabla v_2 \cdot \bgp)(\nabla w_2 \cdot \bgp)\, d\bx \\
    & = 2 \int_{\Omega} 2(\nabla v_2 \cdot \bgp)(\nabla w_2 \cdot \bgp) +  (\nabla v_2 \cdot {\hat \bz})(\nabla w_2 \cdot {\hat \bz})\, d\bx,
\end{split} \]
and for $\bz=z_1 {\hat \bz}$
\[
  a(\bz,\bw)= 2 \int_\Omega (\nabla z_1 \cdot \bgp)(\nabla w_2 \cdot {\hat \bz})\, d\bx.
\]
The problem \eqref{ppp} can be reformulated: find $v_2 \in H^1(\Omega)$ with $\int_\Omega v_2 \, d\bx=0$ such that for all $w_2 \in H^1(\Omega)$:
\begin{equation} \label{pppp}
 \int_\Omega 2(\nabla v_2 \cdot \bgp)(\nabla w_2 \cdot \bgp) +  (\nabla v_2 \cdot {\hat \bz})(\nabla w_2 \cdot {\hat \bz})\, d\bx =- \int_\Omega (\nabla z_1 \cdot \bgp)(\nabla w_2 \cdot {\hat \bz})\, d\bx.
\end{equation} 
Note that due to the assumption \eqref{Gauge}, for $\bv=v_2 \bgp$ as in \eqref{defv} we have $\int_\Omega v_2 \, d\bx=0$. 
The result \eqref{pppp} holds if we have
\begin{equation} \label{k1}
  \nabla v_2 \cdot \bgp=0 \quad \text{on}~\Omega, ~  \nabla v_2 \cdot {\hat \bz}= - \nabla z_1 \cdot \bgp\quad \text{on}~\Omega.
\end{equation}
We check this for the $\bv=v_2 \bgp$, $\bz=z_1 {\hat \bz}$ specified above. First note that
\[
  \nabla v_2 \cdot \bgp= - F'\big({\hat \bz} \cdot \begin{pmatrix} x \\ y \end{pmatrix} \big) {\hat \bz} \cdot \bgp=0, \quad \nabla v_2 \cdot {\hat \bz}= - F'\big({\hat \bz} \cdot \begin{pmatrix} x \\ y \end{pmatrix} \big)
\]
holds. Finally note that
\[
 \nabla z_1 \cdot \bgp = F'\big({\hat \bz} \cdot\begin{pmatrix} x \\ y \end{pmatrix}\big) \nabla \big(\bgp\cdot \begin{pmatrix} x \\ y \end{pmatrix}\big) \cdot \bgp + \bgp \cdot\begin{pmatrix} x \\ y \end{pmatrix} \nabla F'\big({\hat \bz} \cdot\begin{pmatrix} x \\ y \end{pmatrix}\big) \cdot \bgp = F'\big({\hat \bz} \cdot\begin{pmatrix} x \\ y \end{pmatrix}\big).
\]
Hence, the relations in \eqref{k1} hold and thus the proof is completed.
\end{document}